\documentclass{article}

\usepackage[utf8]{inputenc}
\usepackage[T1]{fontenc}
\usepackage{lmodern}
\usepackage{amsfonts,amssymb}
\usepackage[english]{babel}
\usepackage{amsmath}

\usepackage{amsthm}

\usepackage{epsfig}
\usepackage{mathrsfs}
\usepackage{xcolor}
\usepackage{graphicx}
\usepackage{dsfont}
\usepackage{enumerate}
\usepackage{cases}
\usepackage{bbm}
\usepackage[toc,page]{appendix}
\usepackage{hyperref}
\usepackage{stmaryrd}
\usepackage{upgreek}

\def\RR{\mathbb{R}}

\def\supp{\mathrm {supp}\,}

\def\Rx{\mathcal{R}_x }

\def\bO{\bar{\Omega}}
\def\dO{\pa\Omega}

\def\d{\, {\rm{d}} } 
\def\supp{\, {\rm{supp}} }

\def\pa{\partial}
\def\na{\nabla}
\def\eps{\varepsilon}

\def\HSR2s{\mathcal{H}_{\textrm{\tiny{SR}}}^{2s} }

\def\Lsr*{\mathscr{L}_{\textrm{\tiny{SR}}}^*}

\def\Hdiff0{\mathcal{H}_{\textrm{\tiny{diff}},0}^s }

\theoremstyle{plain}
\newtheorem{theorem}{Theorem}[section]
\newtheorem*{theorem*}{Theorem}
\newtheorem{lemma}[theorem]{Lemma}
\newtheorem*{lemma*}{Lemma}
\newtheorem{prop}[theorem]{Proposition}
\newtheorem*{prop*}{Proposition}
\newtheorem{defi}{Definition}[section]
\newtheorem*{defi*}{Definition}

\newtheorem{remark}[theorem]{Remark}
\newtheorem{cor}[theorem]{Corollary}

\makeatletter
\newtheorem*{rep@theorem}{\rep@title}
\newcommand{\newreptheorem}[2]{%
\newenvironment{rep#1}[1]{%
 \def\rep@title{#2 \ref{##1}}%
 \begin{rep@theorem}}%
 {\end{rep@theorem}}}
\makeatother

\newreptheorem{prop}{Proposition}{\bf}{\rm}
\newreptheorem{thm}{Theorem}{\bf}{\rm}
\newreptheorem{cor}{Corollary}{\bf}{\rm}
\newreptheorem{lemma}{Lemma}{\bf}{\rm}

\def\Cmu{C^1_tC^{1,\mu}} 
\def\Cmup{C^1_tC^{1,\mu'}} 
\def\Cmutwo{C^1_tC^{2,\mu}}
\def\Cmuptwo{C^1_tC^{2,\mu'}}

\def\hu{{\hat{U}}}
\def\bu{{\bar{U}}} 

\title{Global well-posedness of Vlasov-Poisson-type systems in bounded domains}
\date{}
\author{Ludovic Cesbron\thanks{Department of Mathematics, ETH Z\"urich, Switzerland}, Mikaela Iacobelli$^*$}

\makeindex
\begin{document}
	\maketitle
	
	\begin{abstract}
	In this paper we prove global existence of classical solutions to the Vlasov-Poisson and the ionic Vlasov-Poisson models in bounded domains. On the boundary, we consider the specular reflection boundary condition for the Vlasov equation and either homogeneous Dirichlet or Neumann conditions for the Poisson equations. 
	\end{abstract}
	
	
\section{Introduction} \label{sec:intro}

In this paper, we investigate the well-posedness of Vlasov-Poisson models in bounded domains. These models describe the evolution of particles in a plasma, which is an ionised gas mostly constituted of two species of charged particles: ions and electrons. Due to the significant difference in size between those two species, the former being much larger and slower than the latter, it is classical to decouple their dynamics.\\
On the one hand, when one investigates the behaviour of electrons it is reasonable to assume that the ions are stationary. Assuming the plasma has low density and that the velocity of the particles is significantly lower than the speed of light -- i.e. neglecting electron-electron collisions and magnetic forces -- one can model the evolution of the distribution function of the electrons $f=f(t,x,v)$,  which represents at time $t$ the probability of finding an electron at position $x$ with velocity $v$, by the following Vlasov-Poisson system: 
\begin{equation} \label{eq:VP}
	(VP) : = \left\{ \begin{aligned}
		&\pa_t f + v\cdot \na_x f + E\cdot\na_v f = 0  &\mbox{ in } (0,+\infty)\times \Omega\times \RR^d,\\
		& E=-\na U, \quad \Delta U = -\rho  & \mbox{ in } (0,+\infty)\times \Omega ,\\
		&f|_{t=0} = f_0  & \mbox{ in } \Omega\times\RR^d 
	\end{aligned} \right. 
\end{equation}
where $\rho = \int_{\RR^d} f \d v$ is the macroscopic density. In this model, the Vlasov equation describes the transport of the electrons under the influence of the electric field $E$, while the Poisson equation models how the electric potential $U$ is generated by the distribution of the electrons. We shall always assume that the initial distribution $f_0$ is non-negative and normalized:
\begin{align} \label{hyp:f0} 
	f_0  \geq 0, \quad \iint_{\Omega\times\RR^d} f_0 \d x \d v = 1.
\end{align}
On the other hand, when one investigates the behaviour of the ions in the plasma, it is common in physics literature to assume that the electrons are close to thermal equilibrium. Indeed, although electron-electron collisions are neglected in the model above because of their rarity, they become relevant in the ion's timescale and it is reasonable to assume that the distribution of the electrons is the thermal equilibrium of a collisional kinetic model. The Vlasov-Poisson model for massless electrons (VPME) -- sometimes called ionic Vlasov-Poisson -- which is a celebrated model for the evolution of ions, can then be derived asymptotically as the ratio of mass between electrons and ions grows small. For more details on the massless limit we refer to \cite{BGNS18}, and for a more thorough introduction of this model we refer e.g. to \cite{GriffinIacobelli20}. The VPME system consists of a Vlasov equation coupled with a non-linear Poisson equation which models how the electric potential is generated by the distribution of the ions and the Maxwell-Boltzmann distribution of the electrons. It reads
\begin{equation} \label{eq:VPME} 
	\left\{ \begin{aligned}
		&\pa_t f + v\cdot \na_x f + E\cdot\na_v f = 0  &\mbox{ in } (0,+\infty)\times \Omega\times \RR^d\\
		& E=-\na U, \quad \Delta U = e^U -\rho - 1  & \mbox{ in } (0,+\infty)\times \Omega \\
		&f|_{t=0} = f_0  & \mbox{ in } \Omega\times\RR^d 
	\end{aligned} \right. 
\end{equation}
with the same assumptions \eqref{hyp:f0} on $f_0$. \\
We consider a bounded $C^{2,1}$ domain $\Omega$ in $\RR^d$ defined via a $C^{2,1}$ function $\xi: \RR^d \to \RR$ as $\Omega = \lbrace x\in \RR^d :\, \xi(x)<0 \rbrace$ and $\dO = \lbrace x\in \RR^d :\, \xi(x) =0 \rbrace$. We assume that $\Omega$ is uniformly convex which means that for some $C_\Omega> 0$ we have
\begin{equation} \label{eq:convexity}
	v \cdot  \na^2 \xi(x)\cdot  v \geq C_\Omega |v|^2  \qquad  \mbox{ for all } (x,v)\in \bO\times\RR^d 
\end{equation}
where $\na^2 \xi$ denotes the Hessian matrix of $\xi$: $v\cdot \na^2 \xi(x) \cdot v = \sum_{i,j} v_i v_j \pa_{ij}\xi(x) $. We also assume that the normal vectors are well defined on the boundary, i.e. $\na \xi(x) \neq 0 $ for any $x$ such that $|\xi(x)| \ll 1$. The outward unit normal vector is then defined, for $x\in \dO$ as: $n(x) =  \na \xi(x) / |\na \xi(x)| $. \\
On the boundary of $\Omega$ we need to prescribe the behaviour of the particles in the Vlasov equation, as well as the behaviour of the electric potential $U$ in the Poisson equation. For the Vlasov part, the boundary condition takes the form of a balance between the in-going and out-going traces of $f$ in the phase-space. Namely, if we introduce the sets $\gamma_\pm = \lbrace (x,v): x\in \dO,\,  \pm v\cdot n(x) >0\rbrace$ and write $\gamma_\pm f$ the restriction of the trace of $f$ to $\gamma_\pm$ then the boundary condition takes the form 
\begin{align*}
	\gamma_- f (t,x,v) = \mathcal{B}[\gamma_+ f](t,x,v) & \quad \mbox{ on } (0,+\infty)\times\gamma_- .
\end{align*}
In this paper, we will focus on the Specular reflection boundary condition: 
\begin{equation} \label{eq:SR}
	\gamma_- f(t,x,v) =  \gamma_+ f (t,x, \Rx v )  \quad \mbox{ on } (0,+\infty)\times\gamma_- 
\end{equation}
with $\Rx v = v-2(v\cdot n(x)) n(x)$. This means that we assume the boundary is a surface with no asperities and the particles bounce on this surface in a billiard-like fashion.  \\
For the Poisson equation on the electric potential $U$, we will consider either the homogenous Dirichlet condition
\begin{equation}  \label{eq:Dirichlet} 
	U(t,x) = 0  \quad \mbox{ on } (0,+\infty)\times\dO
\end{equation}
or the Neumann boundary condition
\begin{equation} \label{eq:Neumann}
	\pa_n U(t,x) = h \quad \mbox{ on } (0,+\infty)\times\dO
\end{equation}
where $\pa_n U = n(x)\cdot \na U$ is the normal derivative of $U$ at $x\in\dO$. Note that we will require $h$ to satisfy a compatibility condition in order for the system to be well-posed. The Dirichlet boundary condition arises when one assumes that the boundary is perfect conductor, and that it is grounded for the homogeneous case that we consider. On the other hand, the Neumann boundary condition comes down to specifying the value of the electric field $E$ everywhere on the surface. We refer e.g. to \cite[Chapter 1.9]{Jackson99} for a more detailed physical interpretation of these conditions. \\

In the case $\Omega=\RR^3$, the Cauchy theory for the Vlasov-Poisson system \eqref{eq:VP} is well developed. In particular, existence and uniqueness of global in time classical $C^1$ solutions has been established in the 90's \cite{Pfaffelmoser1992,Schaeffer91,Horst1993,Lions-Perthame} and there is also a extended literature on weaker notions of solutions, see e.g. \cite{Arsenev75,HorstHunze84,DipernaLions88,AmbrosioColomboFigalli17}. The bounded domain case is more challenging due to the fact that singularities may form at the boundary and propagate inside the domain even in one dimension \cite{Guo95}. In the case of the half-space, the existence of global classical solutions was proved in \cite{Guo94,HwangVelazquez09} for the specular reflection condition \eqref{eq:SR} and both Neumann and Dirichlet conditions on the electric potential. Note that while Y. Guo proved well-posedness in \cite{Guo94} by adapting the velocity moments method of P.L. Lions and B. Perthame \cite{Lions-Perthame}, H.J. Hwang and J.J.L. Vel{\'a}zquez took a different approach in \cite{HwangVelazquez09} by adapting the ideas of Pfaffelmoser \cite{Pfaffelmoser1992} to the half-space case. In both approaches, the key difficulty is the analysis of the trajectories of the particles, governed by the Vlasov equation, near the singular set $\gamma_0$, see \eqref{def:gamma0}, where these transport dynamics are degenerate. H.J. Hwang and J.J.L. Vel{\'a}zquez also refined their approach in order to consider uniformly convex domains of class $C^5$ in \cite{HwangVelazquez10} by means of local changes of coordinates near the singular set $\gamma_0$ that allow them to efficiently estimate the effect of the curvature of the boundary on the transport dynamics. 
In our paper we develop a more global analysis of the trajectories of the particles, without local coordinates, in order to lower the regularity of the boundary to $C^{2,1},$ which is optimal for our notion of classical solutions. Note that weaker notions of solutions have also been investigated in bounded domains, see for instance \cite{Alexandre93,Abdallah94,Weckler95,Mischler99,FernandezReal18}.  \\

For the VPME system, the Cauchy theory is much less developed due to the difficulties arising from the additional non-linearity of the Poisson equation. In the case $\Omega=\RR^3$, global in time weak solutions were first constructed by F. Bouchut in \cite{Bouchut91} and, in one dimension, D. Han-Kwan and M. Iacobelli constructed global weak solutions for measure data with bounded first moment in \cite{HanKwanIacobelli17}. More recently, M. Griffin-Pickering and M. Iacobelli proved the global well-posedness of VPME in the torus in dimensions 2 and 3 \cite{GriffinIacobelli21torus}, and in the whole space in dimension 3 \cite{GriffinIacobelli21R3}. They proved existence of strong solutions for measure initial data with bounded moments -- strong in the sense that if the initial data is $C^1$ then the solution they construct is a classical $C^1$ solution -- and uniqueness of solution with bounded density in the spirit of Loeper's uniqueness result for Vlasov-Poisson \cite{Loeper06}. In this paper we present the first result of well-posedness for VPME in bounded domains.

\section{Main results} \label{sec:mainresults}

Let us first consider the Vlasov-Poisson system \eqref{eq:VP}. We will prove existence and uniqueness of classical solutions in a bounded domain $\Omega$ in dimension $3$ with the boundary conditions mentioned above. One of the key difficulty is to control the behaviour of the solution $f$ near the grazing set 
\begin{align} \label{def:gamma0} 
\gamma_0 := \lbrace (x,v): x\in\dO, \, v\cdot n(x) =0 \rbrace
\end{align}
where the Vlasov equation with specular reflections is degenerate. In order to avoid having singularities at the initial time we shall assume flatness of the initial distribution near the grazing set. Furthermore, we will also assume compactness of support in velocity and regularity of $f_0$. Namely, we consider initial distributions $f_0$ satisfying \eqref{hyp:f0} and
\begin{align} 
	& f_0 \in C^{1,\mu}(\bO\times\RR^3), \quad \mu\in(0,1)  \label{hyp:regf0VP} \\
	& \supp f_0 \subset \subset \bO\times\RR^3  \label{hyp:compactsupportf0VP}\\
	&f_0(x,v) = \mbox{constant}  \quad \mbox{ for all }(x,v) \mbox{ such that } \alpha(0,x,v) \leq \delta_0 \label{hyp:flatnessVP} 
\end{align}
where $\alpha(0,x,v)$ is the kinetic distance defined in Definition \ref{defi:kineticdistance} and Lemma \ref{lem:VelocityLemma}, which measures a distance to the grazing set $\gamma_0$. 
Before stating our existence result for the Vlasov-Poisson system we also need a compatibility condition on $h$ in the Neumann boundary condition case \eqref{eq:Neumann} in order to ensure well-posedness. This condition can be derived by integrating the boundary condition over $\dO$ and using Green's formula and the Poisson equation: 
\begin{equation} \label{eq:Neumanncompatibility} 
	\int_{\dO} h(x) \d x = - \iint_{\Omega\times\RR^3} f_0(x,v) \d x \d v  =-1. 
\end{equation}
Finally, we introduce the following notation: for $\mathcal{A}\subseteq\Omega$ or $\Omega\times\RR^3$ and $T>0$ we write 
\begin{align*}
	\Cmu([0,T]\times\mathcal{A}) := C^1([0,T];C^0(\mathcal{A})) \cap L^\infty((0,T);C^{1,\mu}(\mathcal{A})) .
\end{align*}

We now state our main result for the Vlasov-Poisson system: 
\begin{theorem}\label{thm:mainVP} 
Let $\Omega\subset \RR^3$ be a $C^{2,1}$ uniformly convex domain, $f_0$ satisfy \eqref{hyp:f0}-\eqref{hyp:regf0VP}-\eqref{hyp:compactsupportf0VP}-\eqref{hyp:flatnessVP}, and consider the Vlasov-Poisson system \eqref{eq:VP} with the specular reflection condition \eqref{eq:SR} for the Vlasov equation, and either the Dirichlet boundary condition \eqref{eq:Dirichlet} or Neumann boundary condition \eqref{eq:Neumann} with $h$ satisfying \eqref{eq:Neumanncompatibility} for the Poisson equation. \\
There exists a unique classical solution $f\in \Cmup ((0,\infty)\times\Omega\times\RR^3)$ and $E \in \Cmuptwo((0,\infty)\times\Omega)^3$ for some $\mu' \in(0,\mu)$.
\end{theorem}

Apart from our analysis of the trajectories of transport, our strategy of proof is somewhat classical. We begin in any dimension $d$ with an approximation of the Vlasov-Poisson system by a sequence of linear equations \eqref{eq:iterative}. We show, using our analysis of the trajectories of transport, that given a fixed electric field $E$ the Vlasov equation has a solution in the appropriate functional space. Vice-versa, by classical elliptic regularity we know that given a regular enough density $\rho$ the Poisson equation in $\Omega$ will have a regular solution. This yields sequences of solutions $(f^n)$ and $(E^n)$ to the linear Vlasov and Poisson equations. We then assume boundedness of the velocity uniformly in time, namely that for all $t\in[0,T]$ the quantity $Q^n(t) = \sup \lbrace |v| : (x,v)\in \supp f^n(s),\, 0\leq s\leq t \rbrace$ is bounded by some $K(T)>0$. Under this assumption we show convergence of the sequences $(f^n)$ and $(E^n)$ to a solution of the Vlasov-Poisson system. Then, we restrict ourself to the case $d=3$ and remove the assumption of uniformly bounded velocities by proving that if the velocities are initially bounded \eqref{hyp:compactsupportf0VP} then $Q(t) = \lim_{n\to\infty} Q^n(t)$ is bounded for all $t\in[0,T]$ via a Pfaffelmoser-type argument. Finally, we conclude the proof of Theorem \ref{thm:mainVP} with the global in time existence by showing that the bound on $Q(t)$ holds as $T\to \infty$ and prove uniqueness of solution adapting the idea of P.L. Lions and B. Perthame \cite{Lions-Perthame} albeit in a $L^1$ framework. \\

This strategy of construction of solution via an iterative sequence also applies to the VPME case, however the non-linearity of the Poisson equation in \eqref{eq:VPME} will remain in the iterative sequence therefore classical elliptic regularity will not provide the desired regularity estimates on the force field $E$. In order to derive such estimates we adopt a Calculus of Variation approach, identifying the solution of the Poisson equation with the minimiser of an energy functional which will take into account the boundary conditions. Furthermore, we introduce a splitting of the electric potential into a singular part $\hat{U}$, solution to a linear Poisson equation, and a regular part $\bar{U}$ solution to a non-linear elliptic PDE, in the spirit of \cite{HanKwanIacobelli17,GriffinIacobelli21torus}. The purpose of this splitting is to isolate the difficulties due to the non-linearity from those that we can handle via classical elliptic regularity theory. The estimates we derive, which are stated in Proposition \ref{prop:EllipticRegDir} for the Dirichlet case, and Proposition \ref{prop:EllipticRegNeu} for the Neumann case, are crucial for our analysis and can be useful in the study of singular limits for VPME, as done in \cite{Griffin-PickeringIacobelliJMPA}.
In the Neumann boundary condition case this will naturally lead to a compatibility condition on $h,$ which reads as follows: 
	\begin{align}\label{hyp:hPMENeu} 
h<0, \qquad 	\int_{\dO} |h| \d \sigma(x) < 1+ |\Omega| .
\end{align}
We can now state our main result in the VPME case: 
\begin{theorem}\label{thm:mainVPME} 
	Let $\Omega\subset \RR^3$ be a $C^{2,1}$ uniformly convex domain, $f_0$ satisfy \eqref{hyp:f0}-\eqref{hyp:regf0VP}-\eqref{hyp:compactsupportf0VP}-\eqref{hyp:flatnessVP}, and consider the Vlasov-Poisson model for massless electrons \eqref{eq:VPME} with the specular reflection condition \eqref{eq:SR} for the Vlasov equation, and either the Dirichlet boundary condition \eqref{eq:Dirichlet} or Neumann boundary condition \eqref{eq:Neumann} with $h$ satisfying \eqref{hyp:hPMENeu} for the Poisson equation. \\
	There exists a unique classical solution $f\in \Cmup ((0,\infty)\times\Omega\times\RR^3)$ and $E \in \Cmuptwo((0,\infty)\times\Omega)^3$ for some $\mu' \in(0,\mu)$.
\end{theorem}

For the sake of clarity, we have decided to devote the core of our paper to the Vlasov-Poisson system and treat the VPME case independently in the last section. Since the general method is the same, we will only highlight in that section the differences between the two cases and the modifications required for the VPME case.  \\
The paper is thus organised as follows. In Section \ref{sec:velocitylemma} we develop our global analysis of the trajectories of particles governed by the Vlasov equation in any dimension $d$, which culminates in the Velocity Lemma \ref{lem:VelocityLemma}. In Section \ref{sec:iterative} we focus on the Vlasov-Poisson case and show well-posedness via an approximating sequence of linear problems as explained above and a Pfaffelmoser-like argument in dimension 3 to show  boundedness of the velocities. Finally, in Section \ref{sec:VPME} we derive Elliptic regularity estimates for the non-linear Poisson equation of \eqref{eq:VPME} in any dimension $d$ and outline the proof of well-posedness for the VPME system.  

\section{Velocity Lemma} \label{sec:velocitylemma}

Consider a uniformly convex $C^{2,1}$ domain $\Omega$ and a field $E\in C^{0,1}([0,T]\times\bO)^d$ satisfying $E(t,x) \cdot n(x) >0$ for all $t\in[0,T]$ and $x\in\dO$. The characteristic curves associated to the Vlasov equation with specular reflections $(X_s,V_s)=(X(s;t,x,v), V(s;t,x,v))$ are governed by the following ODE system
\begin{align}
& \pa_s X_s = V_s  & X(t;t,x,v) = x,  \label{eq:characX}\\
& \pa_s V_s = E(s,X_s) & V(t;t,x,v) = v, \label{eq:characV} \\
& V_{s^+} = V_{s^-} - 2 ( n(X_\tau)\cdot V_{s^-} )n(X_s) & \mbox{for all } s \mbox{ s.t. } X_s \in \dO. \label{eq:characSR}  
\end{align}
These trajectories evolve in the phase-space $\bO\times\RR^d$. The purpose of this section is to characterise their distance to the grazing set $\gamma_0,$ and to deduce some control on the number of bounces on $\dO$ that such trajectories undergo in a finite time interval. To that end we begin by defining a notion of kinetic distance 
\begin{defi} \label{defi:kineticdistance}
Consider $\delta>0$ and define the neighbourhood of the boundary 
\begin{align*}
\dO_\delta:= \lbrace x\in\bO: \mbox{dist}(x,\dO)<\delta\rbrace. 
\end{align*} 
We say that $\alpha :[0,T]\times\bO\times\RR^d \to \RR_+$ is a $\delta$-kinetic distance if $\alpha\in C^0([0,T]\times\bO\times\RR^d),$ and satisfies
\begin{align} \label{eq:KTkernel}
\mbox{for all } (t,x,v) \in [0,T]\times\dO_\delta\times\RR^d :\quad   \big[ \alpha(t,x,v) = 0 \big]  \Leftrightarrow \big[ (x,v) \in \gamma_0 \big].
\end{align}
\end{defi}
We then have the following Lemma of isolation of the grazing set.
\begin{lemma} \label{lem:isolatedgrazing}
Consider $E\in C^{0,1}([0,T]\times\bO)^d$ and the flow of transport $( X(s;t,x,v), V(s;t,x,v))$ given by \eqref{eq:characX}-\eqref{eq:characV}-\eqref{eq:characSR}. If there exists a $\delta$-kinetic distance $\alpha$ such that for all $(t,x,v)\in [0,T]\times\dO_\delta\times\RR^d$ and $s\in[0,T]$: 
\begin{align} \label{eq:isolationgamma0}
C^-_s \alpha(t,x,v) \leq \alpha\big( s,X(s;t,x,v), V(s;t,x,v)\big) \leq C^+_s \alpha(t,x,v)
\end{align}
with $C^\pm_s = C^\pm_s(t-s,x,v) >0$, then the grazing set $\gamma_0$ is isolated in the sense that a trajectory $s\to ( X(s;t,x,v), V(s;t,x,v))$ can only reach $\gamma_0$ is if starts in $\gamma_0$. 
\end{lemma}
\begin{proof} 
If $\delta$ is large enough so that $\dO_\delta = \bO$ then the isolation of the grazing set is a direct consequence of \eqref{eq:KTkernel} and \eqref{eq:isolationgamma0} since the former states that $\alpha$ only cancels on $\gamma_0$, while the latter states that $\alpha$ cannot cancel along a trajectory $(X_s,V_s)$ unless it is initially naught. When $\dO_\delta \subsetneq \bO,$ then there may exist $(x,v)\notin\dO_\delta\times\RR^d \supset \gamma_0$ such that $\alpha(0,x,v)=0$. In order to prove Lemma \ref{lem:isolatedgrazing} it is enough to show that a trajectory $(X_s,V_s)$ starting in $\Omega\times\RR^d$ cannot reach $\gamma_0$ without going through $(\dO_\delta\times\RR^d) \setminus \gamma_0$. This follows immediately from the continuity of $s\to X(s;t,x,v)$ given by \eqref{eq:characX}-\eqref{eq:characV}-\eqref{eq:characSR} with $E\in C^{0,1}([0,T]\times\bO)^d$.
\end{proof}
We now turn to the main result of this section: the Velocity Lemma which states the existence of a kinetic distance. 	
\begin{lemma}[Velocity Lemma] \label{lem:VelocityLemma}
Consider a $C^{2,1}$ uniformly convex domain $\Omega$ and a field $E\in C^{0,1}([0,T]\times\Omega)^d$ such that $E(t,x)\cdot\na\xi(x) \geq C_0 >0$ for all $t\in[0,T]$ and $x\in\dO$. We define 
\begin{align}  \label{def:alphadeltakin}
\alpha(s,x,v) =  \frac12 (v\cdot \na \xi(x) )^2 + \big(v\cdot \na^2\xi(x)\cdot v + E(s,x)\cdot\na\xi(x) \big)|\xi(x)| .
\end{align}
There exists $\delta>0$ such that $\alpha$ is a $\delta$-kinetic distance which satisfies \eqref{eq:isolationgamma0} with 
\begin{align*}
C_s^\pm = \exp\big( \pm C_0 \big[ (|v|+1)|s-t| + \lVert E\lVert_{L^\infty}(s-t)^2\big]\big)
\end{align*}
and $C_0=C_0(\| \xi\|_{C^{2,1}}, \|E\|_{C^{0,1}})$.  
\end{lemma}
\begin{proof}
Since $E$ is continuous and $E(t,x)\cdot\na\xi(x) \geq C_0 >0$ for all $t\in[0,T]$ and $x\in\dO$, there exists $\delta>0$ such that $E(t,x)\cdot\na\xi(x)>0$ for all $x\in\dO_\delta$. In that neighbourhood we have, using \eqref{eq:convexity} and the continuity of $\na^2\xi$, that for all $v\in\RR^d$ and $t\in[0,T]$:
\[   v\cdot \na^2\xi(x)\cdot v + E(t,x)\cdot\na\xi(x)  >0,  \]
therefore $\alpha$ only cancels if both $v\cdot \na \xi(x) =0$ and $\xi(x)=0$, i.e. if $(x,v)\in \gamma_0;$ hence $\alpha$ is a $\delta$-kinetic distance. \\
We will prove that $\alpha$ satisfies \eqref{eq:isolationgamma0} by a Gr\"{o}nwall argument, differentiating $\alpha$ along a trajectory. To that end, let us notice that if we write $b = \xi(X_s),$ then $\pa_s b = V_s\cdot\na_x\xi(X_s)$ and $\pa^2_{ss} b = V_s\cdot \na^2\xi(X_s)\cdot V_s + E(s,X_s)\cdot\na\xi(X_s)$ so that we have 
\[ \alpha(s,X_s,V_s) = \frac12 (\pa_s b)^2 - b \pa^2_{ss} b \]
and we easily compute $\tfrac{\d}{\d s} \alpha = -b \pa^3_{sss} b$:
\begin{align*}
\frac{\d}{\d s}\alpha(s,X_s,V_s) &= |\xi(X_s)| \Big( V_s \cdot (V_s\cdot \na^3 \xi(X_s) \cdot V_s) + 3 E(s,X_s)\cdot \na^2 \xi(X_s)\cdot V_s \\
&\qquad +  (\pa_s E(s,X_s)+ V_s\cdot\na_xE(s,X_s))\cdot\na\xi(X_s) \Big) .
\end{align*}
Our regularity assumptions on $E$ and $\xi$ yield on the one hand
\begin{align*}
\left| \frac{\d}{\d s} \alpha(s,X_s,V_s)\right| &\leq C  |\xi(X_s)| \Big( |V_s|^3 \lVert \na^3 \xi \lVert_{L^\infty} + |V_s| \lVert E \lVert_{C_x^{0,1}}  \lVert \xi \lVert_{C^{1,1}} + \lVert \pa_s E\lVert_{L^\infty} \lVert \na \xi\lVert_{L^\infty}  \Big) \\
&\leq C |\xi(X_s)|  ( |V_s|^3+ |V_s| +1)  
\end{align*}
and on the other hand since $\Omega$ is uniformly convex \eqref{eq:convexity}  
\begin{align}
\alpha(s,X_s,V_s) \geq C |\xi(X_s)| ( C_\Omega |V_s|^2 + 1 )
\end{align}
with $C_\Omega >0$. Hence
\begin{align*}
\left| \frac{\d}{\d s} \alpha(s,X_s,V_s)\right|  \leq C (|V_s|+1) \alpha(s,X_s,V_s)
\end{align*}
and Gr\"{o}nwall's Lemma concludes the proof with \eqref{eq:characV}.
%
\end{proof}

\begin{remark} \label{rmk:non-unif} 
Note that if we only assume that $\Omega$ is convex but not uniformly, i.e. $C_\Omega \geq 0$ in \eqref{eq:convexity}, then $\alpha$ given by \eqref{def:alphadeltakin} is still a $\delta$-kinetic distance. However, for $s\leq t,$ the constant $C_s^\pm$ in \eqref{eq:isolationgamma0} given by the Gr\"{o}nwall argument is controlled by
\begin{align*}
C_s^\pm &\leq \exp\bigg(\pm C \bigg[ (|v|+1)|s-t| + \lVert E\lVert_{L^\infty}(s-t)^2+ \int_s^t |V_\tau| \d \tau \bigg] \bigg)\\
&\leq \exp\Big(\pm C \Big[ (|v|+1)|s-t| + \lVert E\lVert_{L^\infty}(s-t)^2+\frac{ (|v|+ \lVert E\lVert_{L^\infty}(s-t))^4}{4\lVert E\lVert_{L^\infty}} - \frac{|v|^4}{4\lVert E\lVert_{L^\infty}} \Big] \Big).
\end{align*}
We will see in Section \ref{subsec:boundonQt} that this bound is not enough to close our proof of existence of classical solutions to the Vlasov-Poisson systems. We believe that the non-uniformly convex domain case actually requires a much finer characterisation of the isolation of the grazing set. 
\end{remark}

The Velocity Lemma is an essential part of this proof because it is directly related with the number of reflections on the boundary that a trajectory $s\to (X(s;t,x,v),V(s;t,x,v))$ undergoes within a given time interval, and consequently with the fact that said trajectory is uniquely defined. Morally, the closer you are to the grazing set the more reflections can happen. We characterise this relation in the following Lemma in which we establish an upper bound on the number of reflections along a trajectory within a given time interval in terms of the distance of that trajectory to the grazing set. 
\begin{lemma}\label{lem:regularitycharacflow}
Under the assumptions of Lemma \ref{lem:VelocityLemma}, for any $(t,x,v)\in [0,T]\times\Omega\times\RR^d$ the trajectory $s\to (X(s;t,x,v), V(s;t,x,v))$ is uniquely defined and the number of reflections $k$ that it undergoes within an interval of time $s\in(t-\Delta,t)$ is bounded above:
\begin{align*}
k \leq \Delta C_1 \frac{(|v|+\Delta \|E\|_{L^\infty})^2+ \|E\|_{L^\infty}}{\sqrt{\alpha(t,x,v)}} e^{C_0 [ (|v|+1) \Delta + \| E\|_{L^\infty}\Delta^2]} 
\end{align*}	
with $C_1=C_1(\Omega) >0$ and $C_0$ is given by Lemma \ref{lem:VelocityLemma}. 
\end{lemma}

%

\begin{proof}
	The fact that the trajectory is uniquely defined follows directly from the upper bound on the number of reflections. Indeed, for any $(t,x,v)\in[0,T]\times(\Omega\times\RR^d)\setminus\gamma_0$, if the trajectory $(X(s;t,x,v),V(s;t,x,v))$ undergoes a finite number of reflections in a finite time, then we can construct the trajectory by the composition of a finite number of transport given by the ODE system \eqref{eq:characX}-\eqref{eq:characV} and specular reflections \eqref{eq:characSR}. The velocity component, $s\to V(s;t,x,v)$ will be piecewise continuous since $E$ is in $C^{0,1}([0,T]\times\bO)^d$, with discontinuities at the reflections times, and since specular reflections do not affect the norm of the velocity, we see that $s\to |V(s;t,x,v)|$ will be continuous. Furthermore, the position component $s\to X(s;t,x,v)$ will be continuous and piecewise $C^1$.\\
	Let us now fix $(t,x,v)\in[0,T]\times(\bO\times\RR^d)\setminus \gamma_0$ and prove the upper bound on the number of reflections. Since $E\in L^\infty([0,T]\times\Omega)$, the norm of the velocity $V(s;t,x,v)$ is uniformly bounded on $(t-\Delta,t)$ as 
	\begin{align*} 
    |V(s;t,x,v)|\leq M= |v|+ \Delta \lVert E\lVert_{L^\infty}.
    \end{align*} 
    Let us consider two consecutive reflection times $t-\Delta \leq s_{i+1}<s_i \leq t$, namely such that $X_{s_i} = X(s_i;t,x,v) \in \dO$, $X_{s_{i+1}} = X(s_{i+1};t,x,v) \in \dO$ and for all $s\in(s_{i+1},s_i)$, $X_s \notin \dO$. Since the velocity is bounded and by continuity of the transport flow on $(s_{i+1},s_i)$, we have immediately 
	\begin{align} \label{eq:boundedXs}
	\frac{|X_{s_{i}}-X_{s_{i+1}} | }{s_{i}-s_{i+1}} < M. 
	\end{align}
	Furthermore, integrating \eqref{eq:characX} and \eqref{eq:characV} over $(s_{i+1},s_{i})$ we get
	\begin{align*} 
	X_{s_{i}} &= X_{s_{i+1}} + \int_{s_{i+1}}^{s_{i}} \Big( V_{s_{i+1}} + \int_{s_{i+1}}^\tau E(u,X_u) \d u \Big)\d \tau \\
	&= X_{s_{i+1}} + (s_{i}-s_{i+1}) V_{s_{i+1}} + \int_{s_{i+1}}^{s_{i}} \int_{s_{i+1}}^\tau E(u,X_u) \d u \d \tau,
	\end{align*}
	which yields, multiplying by $\na\xi (X_{s_{i+1}})$
	\begin{align} \label{eq:Vneargrazing}
	\big| V_{s_{i+1}} \cdot \na\xi (X_{s_{i+1}}) \big| \leq \left| \frac{X_{s_{i}}-X_{s_{i+1}}}{s_{i}-s_{i+1}} \cdot \na\xi (X_{s_{i+1}}) \right| + \frac12  \lVert E \lVert_{L^\infty} (s_{i}-s_{i+1}) .
	\end{align}
	We know that the norm of $\frac{X_{s_{i+1}}-X_{s_i}}{s_{i+1}-s_i}$ is bounded, we are now interested in its direction. To that end, we write the Taylor expansion of $\xi \in C^{2,1}$:
	\begin{equation} 
	\begin{aligned}
	\xi(X_{s_{i}} ) &= \xi( X_{s_{i+1}}) + (X_{s_{i}}-X_{s_{i+1}})\cdot\na\xi(X_{s_{i+1}}) \\
	&\quad +  \frac12 \int_0^1 (X_{s_{i}}-X_{s_{i+1}})\cdot \na^2\xi(X_{s_{i+1}}+ t(X_{s_{i}}-X_{s_{i+1}})) \cdot (X_{s_{i}}-X_{s_{i+1}}) \d t .
	\end{aligned}
	\end{equation}
	Since $\xi$ cancels both at $X_{s_{i+1}}$ and $X_{s_i} \in \dO$, we get
	\begin{align*}
	\big| (X_{s_{i}}-X_{s_{i+1}})\cdot\na\xi(X_{s_{i+1}})\big| \leq \frac12 |X_{s_{i}}-X_{s_{i+1}} |^2 \lVert \na^2\xi\lVert_{L^\infty}.
	\end{align*}
	Together with \eqref{eq:Vneargrazing} and \eqref{eq:boundedXs} this yields 
	\begin{align*}
	\big| V_{s_{i+1}} \cdot \na\xi(X_{s_{i+1}}) \big|  \leq \frac12 \lVert \na^2\xi\lVert_{L^\infty} M^2 (s_{i}-s_{i+1}) + \frac12 \lVert E \lVert_{L^\infty} (s_{i}-s_{i+1}) .
	\end{align*}
	Furthermore, by definition of $\alpha$ \eqref{def:alphadeltakin}, $\alpha(s_{i+1},X_{s_{i+1}}V_{s_{i+1}}) = |V_{s_{i+1}} \cdot \na\xi(X_{s_{i+1}})|^2$ hence we get
	\begin{align*}
	 | s_{i}-s_{i+1} | \geq C \frac{ \sqrt{\alpha(s_{i+1},X_{s_{i+1}}V_{s_{i+1}})} }{M^2+\|E\|_{L^\infty}},
	\end{align*}
	with $C= C( \Omega)$. The result then follows from the Velocity Lemma \ref{lem:VelocityLemma} since $k\leq \sup_i \frac{\Delta}{s_i-s_{i+1}}$ by construction.
\end{proof}

\section{The Vlasov-Poisson system}  \label{sec:iterative}

We will construct a solution to the Vlasov-Poisson equation as a limit of an iterative sequence defined as follows: \\
For any $n\geq 0$ we consider an initial data $f^n_0$ satisfying
\begin{align} 
& f^n_0 \in C^{1,\mu}(\bO\times\RR^d),\quad f^n_0 \geq 0 \label{hyp:regf0} \\
& \supp f^n_0 \subset \subset \bO\times\RR^d  \label{hyp:compactsupportf0}\\
&f^n_0(x,v) = \mbox{constant}  \quad \mbox{ for all }(x,v) \mbox{ such that } \alpha_{n}(0,x,v) \leq \delta_0 \label{hyp:flatness} 
\end{align}
where $\delta_0>0$ is fixed, $\mu\in(0,1)$, and $\alpha_{n}$ is the $\delta$-kinetic distance defined in \eqref{def:alphadeltakin} using the field $E^{n}$ defined below, initiated with the stationary field $ E^0 (x) = -\na U^0$, $\Delta U^0 = -\int f_0 \d v$ with Dirichlet \eqref{eq:Dirichlet} or Neumann \eqref{eq:Neumann}-\eqref{eq:Neumanncompatibility} boundary conditions. Note that the fields $E^n$ are indeed regular enough for this kinetic distance to exist, see Corollary \ref{cor:iterative}. \\
We then define the sequences $f^n$ and $E^n$ for $n\geq 1$ as: 
\begin{equation} \label{eq:iterative} 
\left\{ \begin{aligned}
& \pa_t f^n + v\cdot \na_x f^n + E^{n-1} \cdot \na_v f^n = 0 & \quad \mbox { in } (0,T]\times\Omega\times\RR^d \\
& E^n(t,x) =  -\na U^n, \quad \Delta U^n = - \int_{\RR^d} f^n \d v  &\mbox{ in } (0,T]\times\Omega \\
& f^n(0,x,v) = f^{n-1}_0(x,v) & \mbox{ in } \Omega\times\RR^d
\end{aligned} \right.
\end{equation}
with the specular reflection boundary condition \eqref{eq:SR} for every $f_n$ and either the Dirichlet \eqref{eq:Dirichlet} or the Neumann boundary condition \eqref{eq:Neumann} for every $U_n$, with $h\in C^{1,\mu}(\dO)$ satisfying \eqref{eq:Neumanncompatibility} in the latter case.


\subsection{Well-posedness of the linear problem}

We prove well-posedness of \eqref{eq:iterative} in two steps. First we consider the Vlasov equation with a fixed electric field $E\in C^0_tC^{1,\mu}([0,T]\times\Omega)^d$ and prove existence and uniqueness of a solution $f\in \Cmup([0,T]\times\Omega\times\RR^d)$ for some $\mu'\in(0,\mu)$ in Theorem \ref{thm:LinearVlasov}. Secondly, classical elliptic PDE theory yields the converse, namely the existence and uniqueness of a solution $E\in \Cmutwo([0,T]\times\Omega)^d$ to the Poisson equation for a fixed $\rho \in C^1_tC^{0,\mu}([0,T]\times\Omega)$. Combining these two results, we finally state in Corollary \ref{cor:iterative} the well-posedness of \eqref{eq:iterative}. 

\begin{theorem}\label{thm:LinearVlasov} 
Consider $\Omega$ a $C^{2,1}$ domain and a fixed electric field $E\in C^0_tC^{1,\mu}([0,T]\times\bO)^d$, $\mu\in(0,1]$, satisfying $E(t,x)\cdot\na\xi(x) \geq C_0>0$ for all $t\in[0,T]$ and $x\in\dO$. For all $f_0$ satisfying \eqref{hyp:regf0}-\eqref{hyp:compactsupportf0} and \eqref{hyp:flatness} with the kinetic distance associated with the field $E$, there is a unique solution $f \in \Cmu([0,T]\times\Omega\times\RR^d)$ to the linear Vlasov equation
\begin{equation}\label{eq:linVlasov} 
\left\{ \begin{aligned} 
&  \pa_t f + v\cdot\na_x f + E\cdot\na_v f = 0 \qquad & \mbox{ in }(t,x,v)\in (0,T]\times\bO\times\RR^d ,\\
& \gamma_- f(t,x,v) = \gamma_+ f(t,x, \mathcal{R}_x v) & \mbox{ on } (t,x,v)\in(0,T]\times\gamma_- ,\\
&f|_{t=0} = f_0  & \mbox{ in } \Omega\times\RR^d .
\end{aligned} \right. 
\end{equation}
\end{theorem}
\begin{proof} 
By assumption on $E$ and Lemmas \ref{lem:regularitycharacflow} and \ref{lem:isolatedgrazing} we know that $\alpha$ defined in \eqref{def:alphadeltakin} is a $\delta$-kinetic distance, that the grazing set is isolated and that the flow of transport $(X(s;t,x,v),V(s;t,x,v))$ is uniquely defined. Therefore, there is a unique solution $f$ to our system, which is given by the push-forward of the initial distribution along the flow of transport as expressed by the representation formula:
 \begin{align} \label{eq:representation}
f(t,x,v) = f_0 (X(0;t,x,v), V(0;t,x,v)) .
\end{align}
The key subject of this proof is then the regularity of $f$. Combining the representation formula with the flatness assumption \eqref{hyp:flatness} we see that $f(t,x,v)$ is constant if 
\begin{align*}
\alpha(0,X(0;t,x,v), V(0;t,x,v)) \leq \delta_0,
\end{align*}
and the Velocity Lemma \ref{lem:VelocityLemma} then means that $f(t,x,v)$ is constant if 
\begin{align*}
\alpha(t,x,v) \leq  \delta_0 e^{C_0\big[ (|v|+1)t + \lVert E\lVert_{L^\infty}t^2\big]} \leq \delta_0(T)
\end{align*}
with
\begin{align} \label{def:deltaT}
\delta_0(T) := \delta_0 e^{C_0\big[ (Q+1)T + \lVert E\lVert_{L^\infty}T^2\big]},  \quad Q =  \sup \lbrace |v|,  \, v\in \supp (f_0) \rbrace. 
\end{align}
As a consequence, it is enough to study the regularity of $f$ away from a neighbourhood of $\gamma_0$ where it is constant, i.e. on the set 
\begin{align} \label{def:setO} 
\mathcal{O} = \lbrace (t,x,v) \in [0,T]\times\Omega\times\RR^d:\, \alpha(t,x,v) \geq \delta_0(T) \rbrace .
\end{align}
For any $(t,x,v) \in \mathcal{O}$, we know from Lemma \ref{lem:isolatedgrazing} that the number of reflection $k$ that the trajectory $s\to (X(s;t,x,v),V(s;t,x,v))$ undergoes in the interval of time $[0,t]$ is bounded by  
\begin{align}  \label{eq:uniformKbound}
k \leq k_\delta(T) = TC_1\frac{(Q+T\|E\|_{L^\infty})^2+\|E\|_{L^\infty}}{\sqrt{\delta_0(T)}} e^{C_0 [ (Q+1) t + \| E\|_{L^\infty}t^2]}.
\end{align}
As a consequence the trajectory can be expressed as at most $k_\delta(T)$ compositions of transports inside the domain, governed by the ODEs \eqref{eq:characX}-\eqref{eq:characV}, and specular reflections on the boundary. By classical ODE theory we know that, since $E\in C^0_tC^{1,\mu}([0,T]\times\bO)^d$, the flow of transport inside the domain will be $C^1_tC^{1,\mu}([0,T]\times\Omega\times\RR^d))$. Moreover, at the boundary, by assumption we have $n(x) \in C^{1,1}(\dO)$ so the specular reflection operator $\mathcal{R}_x v = v-2(v\cdot n(x)) n(x)$ is $C^{1,1}(\Gamma_+)$. Thus, the entire flow $(X(s;t,x,v),V(s;t,x,v))$ is $\Cmu([0,T]\times\Omega\times\RR^d)$ for all $(t,x,v)\in\mathcal{O}$, which concludes the proof.  

\end{proof}

We conclude this section with the well-posedness of the sequences $f^n$ and $E^n$
\begin{cor} \label{cor:iterative}
Consider $\Omega$ a $C^{2,1}$ domain and an initial datum $f_0$ satisfying \eqref{hyp:regf0}-\eqref{hyp:compactsupportf0}-\eqref{hyp:flatness} with $\mu\in(0,1]$ . Then the sequences $f^n$ and $E^n$ given by \eqref{eq:iterative} are globally defined on $(0,T)\times\Omega\times\RR^d$ and, moreover, we have for any $T>0$:
\begin{align*}
&f^n \in \Cmu([0,T]\times\Omega\times\RR^d)\\
&E^n \in \Cmutwo([0,T]\times\Omega)^d\\
& \lVert f^n \lVert_{L^\infty(\Omega\times\RR^d)} = \lVert f_0 \lVert_{L^\infty(\Omega\times\RR^d)}, \quad \lVert f^n \lVert_{L^1(\Omega\times\RR^d)} = \lVert f_0 \lVert_{L^1(\Omega\times\RR^d)}
\end{align*}
\begin{proof}
We recall that by standard Elliptic regularity theory, see e.g. \cite[Chapter 6]{GilbargTrudinger} or \cite{Nardi13}, if $f^n \in \Cmu([0,T]\times\Omega\times\RR^d)$ and $\dO$ is $C^{2,1}$ then the field $E^n$ given by \eqref{eq:iterative} is in $\Cmutwo([0,T]\times\Omega)^d$. The well-posedness and the regularity of $f^n$ and $E^n$ follow directly by induction using Theorem \ref{thm:LinearVlasov} and this classical elliptic regularity result. The conservation of the $L^\infty$ norm follows from the representation formula \eqref{eq:representation} and the conservation of the $L^1$ norm follows from integrating the Vlasov equation in \eqref{eq:iterative} over $\Omega\times\RR^d$. 
\end{proof}

\end{cor}

\subsection{Compactness and convergence with bounded velocity} 

In this section we will prove compactness of the sequences $f^n$, $E^n$ under the assumption of bounded velocity support. To that end, we introduce 
\begin{align} \label{def:Qn}
Q^n(t) = \sup \lbrace |v| : (x,v)\in \supp f^n(s),\, 0\leq s\leq t \rbrace
\end{align}
and we shall assume in this section that for all $t\in[0,T]$, $Q^n(t)<K=K(T)$ uniformly in $n$. 
\begin{prop} \label{prop:uniformcontrols} 
Consider $\Omega$ a $C^{2,1}$ domain, and an initial datum $f_0$ satisfying \eqref{hyp:regf0}-\eqref{hyp:compactsupportf0}-\eqref{hyp:flatness} with $\mu\in(0,1]$. Assume that for all $n\in\mathbb{N}$, $Q^n(t) \leq K$. Then, for some $n_0>0$, the sequences $f^n$ and $E^n$ given by \eqref{eq:iterative} satisfy for all $n\geq n_0$: 
\begin{align*}
\lVert E^n \lVert_{C^1_tC^{2,\mu}_x} \leq C(T) \\
\lVert f^n \lVert_{C^1_tC^{1,\mu}_{x,v}} \leq C(T). 
\end{align*}
where $C(T)$ depends only on $T$, $K$ and $\lVert f_0\lVert_{L^\infty}$. 
\end{prop}
\begin{proof}
From the uniform bound on $Q^n(t)$ and the conservation of $L^\infty$ norm in Corollary \ref{cor:iterative} we have 
\begin{align*}
\lVert \rho^n(t) \lVert_{L^\infty(\Omega)}  &= \sup_\Omega \int_{\RR^d} f(t,x,v)\d v \\
&\leq \lVert f(t) \lVert_{L^\infty(\Omega\times\RR^d} Q^n(t)^d \\
&\leq K^d \lVert f_0 \lVert_{L^\infty(\Omega\times\RR^d)}. 
\end{align*}
Hence, by classical elliptic regularity, $E^n$ is log-Lipschitz uniformly in $n$. This allows for uniform estimates on $\rho^{n+1}$ in $C^{0,\gamma}(\Omega)$ for some $\gamma<1$. Indeed, we can consider $v,w\in\RR^d$, $x,y\in\Omega$ and $s_*^{n+1}\in[0,t]$ such that for all $s\in (s_*^{n+1},t)$ trajectories $X^{n+1}(s;t,x,v)$ and $X^{n+1}(s;t,y,v)$ do not undergo any reflections on the boundary. Then, introducing 
\begin{align*}
Y^{n+1}(s) = | X^{n+1}(s;t,x,v) - X^{n+1}(s;t,y,w) | + |V^{n+1}(s;t,x,v) - V^{n+1}(s;t,y,w) | 
\end{align*}
we have, using the characteristic equations
\begin{align*}
| \dot{Y}^{n+1}(s)| &\leq |V^{n+1}(s;t,x,v) - V^{n+1}(s;t,y,w) | + |E^n(s,X(s;t,x,v)) - E^n(s;X(s;t,y,w))| \\
&\leq Y^{n+1}(s) - C | X^{n+1}(s;t,x,v) - X^{n+1}(s;t,y,w) | \log \big(| X^{n+1}(s;t,x,v) - X^{n+1}(s;t,y,w) |\big)\\
&\leq CY^{n+1}(s)  \big|  \log Y^n(s)\big| 
\end{align*}	
hence 
\begin{align*}
Y^{n+1}(s) \leq e^{\log Y^{n+1}(t)e^{-C(t-s)}}  \leq( Y^{n+1}(t))^\gamma \leq |x-y|^\gamma + |v-w|^\gamma 
\end{align*}
for all $\gamma \leq e^{-C(t-s^{n+1}_*)}$. We have proved that the flow $(X^{n+1}(s;t,x,v),V^{n+1}(s;t,x,v))$ is uniformly bounded in some $C^{0,\gamma}(\Omega\times\RR^d)$ for $s\in (s^{n+1}_*,t)$.\\ 
Analogously to the proof of Theorem \ref{thm:LinearVlasov} we introduce the set 
\begin{align} \label{def:setOn}  
\mathcal{O}_{n} = \lbrace (t,x,v) \in [0,T]\times\Omega\times\RR^d : \alpha_{n}(x,v) \geq \delta(T) \rbrace
\end{align}
with $\delta(T)$ given by \eqref{def:deltaT} and we know that $f^{n+1}$ is constant on $(\bO\times\RR^d)\setminus \mathcal{O}_n$ so it is enough to study its regularity on $\mathcal{O}_n$. \\
Note that $\delta(T)$ does not depend on $n$ thanks to the uniform bounds on $Q^n(t)$ and $\| E^n\|_{L^\infty}$. Similarly for any $(t,x,v)\in\mathcal{O}_n$ we can also bound the number of reflections within the interval $(0,t)$ uniformly in $n$ with $k_\delta(T)$ given by \eqref{eq:uniformKbound}. Hence, the flow $(X^{n+1}(s;t,x,v),V^{n+1}(s;t,x,v))$ can be expressed as at most $k_\delta(T)$ compositions of transports in $C^{0,\gamma}(\Omega\times\RR^d)$ and specular reflections, hence the flow is in $C^{0,\gamma'}(\Omega\times\RR^d)$ with $\gamma'\leq \gamma^{k_\delta(T)} \leq e^{-k_\delta(T)CT}$. Combined with the representation formula \eqref{eq:representation} this yields a uniform bound of $f^{n+1}$, and in turns $\rho^{n+1}$, in $C^{0,\gamma'}(\Omega\times\RR^d)$ and $C^{0,\gamma'}(\Omega)$ respectively. \\
Consequently, $E^{n+1}$ will be in $C^{1,\gamma'}(\Omega)^d$ by classical elliptic regularity, which means the system \eqref{eq:iterative} at rank $n+2$ satisfies the assumptions of Theorem \ref{thm:LinearVlasov} and our proposition then follows by iteration. 

\end{proof}

\begin{remark}
Note that the limiting factor for the regularity is, in fine, the specular reflection operator which stops the flow of transport from reaching any regularity above $C^{1,1}$. 
\end{remark}

Using the uniform bounds derived in the previous section, we now state the following convergence result. 
\begin{prop} \label{prop:limitsolutions}
Consider $\Omega$ a $C^{2,1}$ domain, and an initial datum $f_0$ satisfying \eqref{hyp:regf0}-\eqref{hyp:compactsupportf0}-\eqref{hyp:flatness} with $\mu\in(0,1]$. Assume that for all $n\in\mathbb{N}$, $Q^n(t) \leq K$. Then as $n\to \infty$: 
\begin{align*}
(f^n,E^n) \to (f,E) \quad \mbox{in } C^{\nu}_tC^{1,\mu}([0,T]\times\Omega\times\RR^d) \times C^\nu_tC^{2,\mu}([0,T]\times\Omega)^d
\end{align*}
with $0<\nu<1$. Moreover the limits $f$ and $E$ are in $\Cmu([0,T]\times\Omega\times\RR^d)$ and $C^1_tC^{2,\mu}([0,T]\times\Omega)^d$ respectively, and $(f,E)$ is solution to the Vlasov-Poisson system \eqref{eq:VP}-\eqref{eq:SR} with either the Dirichlet \eqref{eq:Dirichlet} or the Neumann \eqref{eq:Neumann}-\eqref{eq:Neumanncompatibility} boundary condition.
\end{prop}
Unlike the previous results of this section, this proof does not rely on any geometrical considerations but rather on some standard functional analysis so we refer to \cite[Proposition 3]{HwangVelazquez10} for the proof of this Proposition.
%

\subsection{Global bound on $Q(t)$ in dimension 3}  \label{subsec:boundonQt}

We now restrict ourselves to the 3-dimensional case. The purpose of this section is to remove the assumption of bounded velocity support by proving that if $f_0$ is compactly supported in velocity \eqref{hyp:compactsupportf0VP}, i.e. $Q(0)<K$, then 
\begin{align} \label{def:Qt} 
Q(t) =  \sup \lbrace |v| : (x,v)\in \supp f(s),\, 0\leq s\leq t \rbrace
\end{align}
is uniformly bounded on $[0,T]$.
\begin{prop} \label{prop:Pfaffelmoser} 
Consider $\Omega$ a $C^{2,1}$ domain, an initial datum $f_0$ satisfying \eqref{hyp:regf0}-\eqref{hyp:compactsupportf0}-\eqref{hyp:flatness} with $\mu\in(0,1]$ and the solution $(f,E)\in \Cmu([0,T]\times\Omega\times\RR^3)\times C^1_tC^{2,\mu}([0,T]\times\Omega)^3$ of the Vlasov-Poisson system given by Proposition \ref{prop:limitsolutions}. Then there exists $K(T)<\infty$ depending only on $T$ and $f_0$ such that 
\begin{align*}
Q(t) \leq K(T) \qquad \mbox{ for all } t\in [0,T]. 
\end{align*}
\end{prop}
\begin{proof} 
We prove this proposition via an estimation of the acceleration of the velocity along a characteristic trajectory. To that end let us fix a trajectory $(\hat{X}(t),\hat{V}(t))$. We wish to control, for some $\Delta>0$:
\begin{align}
\int_{t-\Delta}^t | E(s,\hat{X}(s)) | \d s &\leq \int_{t-\Delta}^t \iint_{\Omega\times\RR^3} \frac{f(s,y,w)}{|y-\hat{X}(s)|^2} \d y \d w \d s + C \Delta  \| h \|_{L^\infty} \label{eq:Pfaffelyw}\\
&\leq \int_{t-\Delta}^t \iint_{\Omega\times\RR^3} \frac{f(t,x,v)}{|X(s;t,x,v)-\hat{X}(s)|^2} \d x \d v \d s + C \Delta  \| h \|_{L^\infty}  \label{eq:Pfaffelxv}
\end{align}
where we used the fact that the evolution of the characteristic trajectories is Hamiltonian, hence $\d y \d w = \d X(t;s,y,w) \d V(t;s,y,w) = \d x\d v$. \\
Following the approach of Pfaffelmoser \cite{Pfaffelmoser1992}, we will split this integral in three parts. We introducing three parameters $\eta,\beta,\gamma>0$ to be determined later on, and define $P=Q^\eta$, $R=Q^\beta$ and $\Delta = c_0Q^{-\gamma}$, with $c_0>0$ fixed. We then split the domain of integration as follows:
\begin{align}
& G= \lbrace (s,x,v) \in [t-\Delta,t]\times\Omega\times\RR^3 : \, |v|<P \mbox{ or } |v-\hat{V}(t) |<P \rbrace \\
& B = \lbrace (s,x,v) \in  ([t-\Delta,t]\times\Omega\times\RR^3) \setminus G :\, |X(s;t,x,v)-\hat{X}(s)| < \eps_0(v) \rbrace \\
& U = \lbrace (s,x,v) \in  ([t-\Delta,t]\times\Omega\times\RR^3) \setminus G :\, |X(s;t,x,v)-\hat{X}(s)| > \eps_0(v) \rbrace 
\end{align}
with
\begin{align} \label{def:eps0}
 \eps_0(v) = \max\left\{ \frac{R}{|v|^3} , \frac{R}{|v-\hat{V}(t)|^3}  \right\}. 
\end{align}
We shall now handle each part of the integration individually. Throughout this section, we are mostly interested in bounding each part of the integral with respect to powers of $Q$ and the constants in these bounds will not play a role so we introduce the notation $a \lesssim b$ to denote $a\leq Cb$ for some constant $C$ independent of $Q$. Note that we will also commonly use the notation $G$, $B$ or $U$ to mean subset of $\Omega\times\RR^d$ when the time parameter $s$ is fixed, and also write $G$ as a subset of $\RR^d$ for fixed $(s,x)$. 

\begin{remark}
Note that our definition of $\eps_0$ defers from the one of Hwang and Velazquez in \cite{HwangVelazquez10} which morally includes the term $R/|v-\hat{V}^+(t)|^3$ in the maximum \eqref{def:eps0}, with $\hat{V}^+(t)$ the specular reflection of the $\hat{V}(t)$ at the last time of reflection $s_0\in[t-\Delta,t]$. \\
This difference is directly related to the fact that we develop in the paper a global analysis of the trajectories, whereas Hwang and Velazquez developed a localised analysis. In the global framework, the main argument to control the evolution of the velocity is the balance between the number of reflections and the impact of these reflections on the direction of the velocity. Heuristically, the more reflections happen within the interval $(t-\Delta, t)$, the closer the trajectory is to the grazing, i.e. the more tangential the trajectory is at the points of reflections, hence the lesser the impact of the specular reflection on the direction of the velocity. In this context, it is enough to compare $v$ with $\hat{V}(t)$ so we do not need to add a comparison with $\hat{V}_+(t)$ in the definition of $\eps_0.$
\end{remark}

\subsubsection{The Good Set}
For the integral over the good set $G$, the key argument is the following pointwise control of $E$: for any $(s,x)\in(0,T)\times\Omega$ and $\lambda>0$ 
\begin{align*}
| E(s,x) | &\lesssim \int_\Omega \frac{\rho(s,y)}{|x-y|^2} \d y \\
&\lesssim \| \rho(s)\|_{L^\infty} \int_{|x-y|<\lambda} \frac{\d y }{|y-x|^2} + \| \rho(s)\|_{L^{5/3}} \left(\int_{|y-x|>\lambda} \frac{\d y}{|y-x|^5} \right)^{2/5} \\
&\lesssim \| \rho(s)\|_{L^\infty} \lambda +  \| \rho(s)\|_{L^{5/3}} \lambda^{-4/5} .
\end{align*}
Choosing $\lambda$ such that $\| \rho(s)\|_{L^\infty} \lambda = \| \rho(s)\|_{L^{5/3}} \lambda^{-4/5} $ yields
\begin{align*}
|E(s,x)| \lesssim \| \rho(s)\|_{L^\infty}^{4/9} \| \rho(s)\|_{L^{5/3}}^{5/9} .
\end{align*}
Moreover, the norm  $ \| \rho(s)\|_{L^{5/3}}$ is related to the kinetic energy of the system: for any $\lambda>0$
\begin{align*}
\rho(s,x) \leq \int_{|v|<\lambda} f \d v + \lambda^{-2} \int_{|v|>\lambda} |v|^2 f \d v \lesssim \| f_0\|_{L^\infty} \lambda^3 + \lambda^{-2} \int |v|^2 f \d v
\end{align*}
and we know that the kinetic energy is bounded: $ \iint |v|^2 f(t) \d v\d x := K(t) \leq K <\infty$ by conservation of the total energy of the system, see e.g. \cite[Chapter 4]{Glassey}. Hence choosing $\lambda = K^{1/5}$ yields $\| \rho(s)\|_{L^{5/3}} \lesssim K^{3/5}$. Therefore we can bound $E$ pointwise as
\begin{align}
| E(s,x)| &\lesssim \| \rho\|_{L^\infty}^{4/9} \lesssim \left( \int_{\RR^3} f(s,x,v) \d v \right)^{4/9} \lesssim \left( \| f_0\|_{L^\infty} \int_{|v|<Q} \d v \right)^{4/9} \nonumber \\
&\lesssim Q^{4/3} .\label{eq:EQ43} 
\end{align}
We use this bound on $E$ to derive a bound on $w= V(s;t,x,v)$ when $|v|<P$: 
\begin{align*}
|w| = |v| + \int_s^t |E(u,X(u))| \d u \leq P+Q^{4/3} (t-s) .
\end{align*}
Choosing $s\in(t-\Delta,t)$ and $ \gamma \geq 4/3 - \eta \geq 0$  we get $|w| \lesssim P$. This yields a control of the restriction of $\rho$ to the good set $G$: 
\begin{align*}
\rho_G(s,y) := \int_{w\in G}  f(s,y,w)\d w \lesssim \|f_0\|_{L^\infty} P^3 .
\end{align*}
Let us now use these estimates to control the integral over the good set of \eqref{eq:Pfaffelyw} with an approach similar to that of the pointwise bound of $E$ above. Since the integral w.r.t to $s$ will not play a role in this bound we first control the rest: for any $\lambda>0$:
\begin{align*}
\iint_G  \frac{f(s,y,w)}{|y-\hat{X}(s)|^2} \d y \d w  &\lesssim  \int_\Omega \frac{\rho_G(s,y)}{|y-\hat{X}(s)|^2}  \d y \\
&\lesssim \|\rho_G(s)\|_{L^\infty} \int_{|y-\hat{X}(s)|<\lambda} \frac{1}{|y-\hat{X}(s)|^2} \d y + \| \rho_G\|_{L^{5/3}} \int_{|y-\hat{X}(s)|>\lambda} \frac{1}{|y-\hat{X}(s)|^2} \d y \\
&\lesssim \|\rho_G(s)\|_{L^\infty} \lambda +  \| \rho_G\|_{L^{5/3}} \lambda^{-4/5} \\
&\lesssim \|\rho_G(s)\|_{L^\infty}^{4/9}  \| \rho_G\|_{L^{5/3}}^{5/9} 
\end{align*}
where we chose $\lambda$ such that $\|\rho_G(s)\|_{L^\infty} \lambda = \| \rho_G\|_{L^{5/3}} \lambda^{-4/5}$. Note that $\rho_G(s,y)\leq \rho(s,y)$ by positivity of $f$, hence we have immediately $\| \rho_G\|_{L^{5/3}} \leq \| \rho \|_{L^{5/3}} \leq K$ the kinetic energy. Finally, we get 
\begin{align*}
\iiint_G \frac{f(s,y,w)}{|y-\hat{X}(s)|^2} \d y \d w \d s \lesssim \Delta P^{4/3}. 
\end{align*}

\subsubsection{The Bad Set}
The control of the integral over the bad set $B$ in \eqref{eq:Pfaffelxv} follows rather immediately from our choice of $\eps_0$: 
\begin{align*}
\iiint_B \frac{f(t,x,v)}{|X(s;t,x,v)-\hat{X}(s)|^2} \d x\d v \d s &\lesssim \int_{t-\Delta}^t \int_{v\notin G} \| f\|_{L^\infty} \eps_0(v) \d v \d s \\
&\lesssim \Delta  \int_{v\notin G} \max\left\{ \frac{R}{|v|^3} , \frac{R}{|v-\hat{V}(t)|^3} , \frac{R}{|v-\hat{V}^+(t)|^3} \right\} \d v \\
&\lesssim \Delta  R \int_P^{Q} \frac{1}{r} \d r \\
&\lesssim \Delta R \ln\left(\frac{Q}{P}\right).
\end{align*}

\subsubsection{The Ugly Set} \label{subsec:ugly} 

For the ugly set $U$, the time-integration in \eqref{eq:Pfaffelxv} is essential. Let us fix some $(t,x,v)$ and define $\mathcal{W}(s)$ as
\begin{align*}
\mathcal{W}(s)= V(s;t,x,v) -  v. 
\end{align*} 
We easily deduce from our transport dynamics \eqref{eq:characX}-\eqref{eq:characV}-\eqref{eq:characSR} the following system of ODEs for the evolution of $\mathcal{W}(s)$: 
\begin{align} 
\dot{\mathcal{W}}(s) &= E(s,X(s;t,x,v)) , & \mathcal{W}(t)=0, \label{eq:characW}\\
\mathcal{W}(\tau) &= \mathcal{R}_{X(\tau;t,x,v)}V(\tau; t,x,v) - v & \mbox{for all } \tau \mbox{ such that } X(\tau;t,x,v)\in\dO, \label{eq:characSRX} 
\end{align}
Note that the reflection condition \eqref{eq:characSRX} does not preserve the norm of $\mathcal{W}(\tau)$, as such it is not comparable to specular reflections and the norm $|\mathcal{W}(s)|$ will not be a continuous function of $s$. Nevertheless, we can bound the jump of $|\mathcal{W}(s)|$ at a reflection time $\tau$ using the Velocity Lemma \ref{lem:VelocityLemma}. Indeed, recall that for any $\tau \in (t-\Delta,t)$, if $(x,v)$ belongs to $\dO_\delta$ as defined in Lemma \ref{lem:isolatedgrazing}  then the Velocity Lemma \ref{lem:VelocityLemma} yields 
\begin{align*}
| V(\tau;t,x,v)\cdot n(X(\tau;t,x,v))| &\leq \sqrt{\alpha(\tau,X(\tau;t,x,v),V(\tau;t,x,v))} \\
&\leq \left( \alpha(t,x,v) e^{C(|v|+1)(t-\tau) + \|E\|_{L^\infty} (t-\tau)^2} \right)^{1/2} \\
&\leq  \left( \alpha(t,x,v) e^{C[(Q+1)\Delta + Q^{4/3}\Delta^2]}\right)^{1/2} 
\end{align*}
hence   
\begin{align}
|\mathcal{W}(\tau^+) | &= | V(\tau^-;t,x,v) - v - 2(V(\tau^-;t,x,v)\cdot n(X(\tau;t,x,v))) n(X(\tau;t,x,v)) | \nonumber\\
&\leq |\mathcal{W}(\tau^-)| + 2 \left( \alpha(t,x,v) e^{C[(Q+1)\Delta + Q^{4/3}\Delta^2]} \right)^{1/2}. \label{eq:Ws1ref}
\end{align}
Moreover, from the uniform bound on $E$ \eqref{eq:EQ43} we know that $\mathcal{W}$ is Lipschitz between reflection times so that, assuming there are $k$ reflections within the interval $(t-\Delta,t)$ we have
\begin{align}
|\mathcal{W}(s)| \lesssim  Q^{4/3}\Delta + 2k \left( \alpha(t,x,v) e^{C[(Q+1)\Delta + Q^{4/3}\Delta^2]} \right)^{1/2}. \label{eq:Wskref}
\end{align}
If $x\notin \dO_\delta$, then the coefficient before $|\xi(x)|$ in \eqref{def:alphadeltakin} could take negative values and $\alpha(t,x,v)$ could cancel even though $(x,v)\notin \gamma_0$. However, if $\tau_1\in(t-\Delta,t)$ is the first reflection time of the backwards trajectory $(X(s;t,x,v),V(s;t,x,v))$ then, by continuity of $X(s)$, the trajectory will reach $\dO_\delta$ at some time $s=t-\tilde{\delta}_1 \in (\tau_1,t)$ before it reflects on the boundary. In the interval $[t-\tilde\delta_1,t]$ there are no reflections so the evolution of $|\mathcal{W}(s)|$ is bounded by the uniform estimate of the electric field \eqref{eq:EQ43}. As a consequence, the inequality \eqref{eq:Ws1ref} still holds with $\alpha(t-\tilde{\delta},X(t-\tilde{\delta};t,x,v), V(t-\tilde{\delta};t,x,v))$ instead of $\alpha(t,x,v)$. Moreover, the same argument holds at any reflection time $\tau_i \in (t-\Delta, t)$, namely there exists a $\tilde{\delta}_i>0$ such that for $s\in (\tau_i - \tilde{\delta}_i, \tau_i)$, $X(s;t,x,v)\in \dO_\delta$ and the isolation of the grazing set ensures that 
\begin{align*}
&\alpha( \tau_i-\tilde{\delta}_i, X(\tau_i-\tilde{\delta}_i; t,x,v),  V(\tau_i-\tilde{\delta}_i; t,x,v)) \\
&\qquad \leq \alpha(t-\tilde{\delta}, X(t-\tilde{\delta};t,x,v), V(t-\tilde{\delta};t,x,v)) e^{C[(Q+1)(\tau_i-\tilde{\delta}_i-t+\tilde{\delta})+ Q^{4/3}(\tau_i-\tilde{\delta}_i-t+\tilde{\delta})^2} \\
&\qquad \leq \alpha(t-\tilde{\delta}, X(t-\tilde{\delta};t,x,v), V(t-\tilde{\delta};t,x,v)) e^{C[(Q+1)\Delta + Q^{4/3}\Delta^2]}.
\end{align*}
Since \eqref{eq:Ws1ref} is established at a reflection time $\tau$, this control of the kinetic distance holds and \eqref{eq:Wskref} follows. Therefore, we define an extension of the kinetic distance as 
\begin{equation*}
\alpha(t,x,v) := \left\{ \begin{aligned} & \alpha(t,x,v) \qquad & \forall x\in \dO_\delta \\
																& \alpha(t-\tilde{\delta}, X(t-\tilde{\delta};t,x,v), V(t-\tilde{\delta};t,x,v)) & \forall x \notin \dO_\delta \end{aligned} \right. 
\end{equation*}
and \eqref{eq:Wskref} holds without the need to distinguish the cases $x\in\dO_\delta$ and $x\notin \dO_\delta$. \\
Now, we know from Lemma \ref{lem:regularitycharacflow} that $\alpha(t,x,v)$ yields a upper bound on the number of reflections $k$. More precisely, with $|v|+\Delta\|E\|_{L^\infty} \sim Q$ (c.f. the proof of Lemma \ref{lem:regularitycharacflow}), and using \eqref{eq:EQ43} we have, assuming w.l.o.g. that $Q\geq 1$
\begin{align*}
k \lesssim  \frac{Q^2 \Delta}{\sqrt{\alpha(t,x,v)}} e^{C [ (Q+1) \Delta + Q^{4/3}\Delta^2]} 
\end{align*}
from which we deduce the following bound on $|\mathcal{W}(s)|$: 
\begin{align*}
|\mathcal{W}(s)| \lesssim  Q^{4/3} \Delta + Q^2 \Delta e^{C [ (Q+1) \Delta + Q^{4/3}\Delta^2]}.
\end{align*}
In order for the exponential term to be uniformly bounded in $Q$, we want $\Delta Q$ to decay when $Q$ grows large. Therefore, we choose $\gamma>1$ and get the following estimate
\begin{align}\label{eq:Ws}
|\mathcal{W}(s)| \lesssim c_0 Q^{2-\gamma} \lesssim c_0 P Q^{2-\gamma-\eta} \leq \frac14 P,
\end{align}
if we choose $1< \gamma\leq 2-\eta$ with $\eta<3/4$ and the appropriate choice of $c_0>0$ in the definition of $\Delta= c_0Q^{-\gamma}$. \\
Now that we've established this bound on $|\mathcal{W}(s)|$, the control of the integral over the ugly set in \eqref{eq:Pfaffelxv} follows rather classically. As a consequence, we shall skip the details which can be found e.g. in \cite{Schaeffer91} or \cite[Section 4.4]{Glassey} and only outline the following steps of the proof. First, note that one can define $\hat{\mathcal{W}}(s) = \hat{V}(s)-\hat{V}(t)$ and derive the same estimate so that we have
\begin{align*}
|V(s;t,x,v) - \hat{V}(s)| &\geq |v-\hat{V}(t)| - |V(s;t,x,v)-v| - |\hat{V}(t)-\hat{V}(s)| \\
&\geq  |v-\hat{V}(t)| - \frac12 P\\
&\geq \frac12  |v-\hat{V}(t)|
\end{align*}
where the last bound follows from the fact that $(s,x,v)\in U$. We now define $Z(s)$ as
\begin{align*}
Z(s) = X(s;t,x,v) - \hat{X}(s)
\end{align*}
and have immediately $|\dot{Z}(s)| = |V(s;t,x,v) - \hat{V}(s)| \geq |v-\hat{V}(t)|/2$. Moreover, one can compare $Z(s)$ with the linear approximation $\bar{Z}(s) = Z(s_0) +\dot{Z}(s_0)(s-s_0)$ with $s_0\in[t-\Delta,t]$ minimizing $|Z(s)|$, and using the uniform bound on $\ddot{Z}(s)$ which follows from \eqref{eq:EQ43} one gets
\begin{align*}
|Z(s)|  \gtrsim |s-s_0| |v-\hat{V}(t)| .
\end{align*}
Furthermore, considering $(s,x,v)\in U$ for which $\eps_0(v) = \frac{R}{|v|^3}$, the substitution $s\to \tau = |s-s_0| |v-\hat{V}(t)|$ yields
\begin{align*}
\int_{t-\Delta}^t \frac{\d s}{|Z(s)|^2} &\leq \int_{t-\Delta}^t \frac{\d s}{ |s-s_0|^2 |v-\hat{V}(t)|^2} \\
&\leq \frac{1}{|v-\hat{V}(t)|} \bigg( \int_0^{\eps_0(v)} \frac{1}{\eps_0^2} \d \tau + \int_{\eps_0}^{+\infty} \frac{1}{\tau^2} \d \tau   \bigg)\\
&\lesssim \frac{|v|^3}{R|v-\hat{V}(t)|}.
\end{align*}
Similarly, for $(s,x,v)\in U$ such that $\eps_0(v) = \frac{R}{|v-\hat{V}(t)|^3}$ one gets
\begin{align*}
\int_{t-\Delta}^t \frac{\d s}{|Z(s)|^2} \lesssim \frac{|v-\hat{V}(t)|^3}{R|v-\hat{V}(t)|}
\end{align*}
so  that for any $(s,x,v)\in  U$:
\begin{align*}
\int_{t-\Delta}^t \frac{\d s}{|Z(s)|^2} &\lesssim \frac{1}{R|v-\hat{V}(t)|}  \Big( \min\left\{|v|,|v-\hat{V}(t)|\right\} \Big)^3 \lesssim \frac{|v|^2}{R} . 
\end{align*}
Using the boundedness of the kinetic energy, we finally derive the following estimate
\begin{align*}
\iiint_U \frac{f(t,x,v)}{|X(s;t,x,v)-\hat{X}(s)|^2} \d x\d v \d s  \lesssim \frac{1}{R} \iint_U |v|^2 f(t,x,v) \d x \d v \lesssim \frac{1}{R}. 
\end{align*}

\subsubsection{Conclusion of the Pfaffelmoser argument} \label{subsec:conclusionPfaffelmoser}

Collecting the estimates above, we have
\begin{align*}
\frac{1}{\Delta} \int_{t-\Delta}^t | E(s,\hat{X}(s)) | \d s &\lesssim P^{4/3} +  R \ln \left( \frac{Q}{P} \right) + \frac{1}{\Delta R} \\
&\lesssim Q^{4\eta/3} + Q^\beta \ln( Q^{1-\eta}) + Q^{\gamma-\beta} .
\end{align*}
To optimize the order of $Q$ on the right hand side we can take $R=Q^\beta \ln^{-1/2}(Q)$ and choose $\eta=\frac{6}{11}$, $\beta=\frac{8}{11}$ and $\gamma=\frac{16}{11}$ which yields 
\begin{align*}
\frac{1}{\Delta} \int_{t-\Delta}^t | E(s,\hat{X}(s)) | \d s &\lesssim Q^{8/11} \ln^{1/2} Q .
\end{align*}
We see in particular that the right-hand-side is sub-linear in $Q$ so we have proved that 
\begin{align} \label{eq:bounddotQ}
\frac{ Q(t) - Q(t-\Delta)}{\Delta} \lesssim Q^{8/11}(t) \ln^{1/2} Q(t) 
\end{align}
and the boundedness of $Q$ follows from a classical iteration procedure, see e.g. \cite[Section 4.5]{Glassey}. 

\end{proof} 

\subsection{Proof of Theorem \ref{thm:mainVP}} \label{sec:proofthmVP}
 
\begin{proof}[Proof of Theorem \ref{thm:mainVP}] 
In the previous sections, we have constructed solutions to the Vlasov-Poisson system in the sense of Theorem \ref{thm:mainVP} on a time interval $[0,T]$ in dimension 3. Note, indeed, that due to the Hopf Lemma the electric field in the Dirichlet case satisfies $E(t,x)\cdot n(x)<0$ for all $t\in[0,T]$ and $x\in\dO$ therefore the Velocity Lemma 
\ref{lem:VelocityLemma} applies. To conclude the proof of the theorem we only need to show uniqueness and that this construction holds as $T\to \infty$. 

\paragraph{Global solutions: }
First of all, recall that thanks to the flatness assumption \eqref{hyp:flatnessVP} and the isolation of the grazing set given by Lemma \ref{lem:VelocityLemma} we can restrict our analysis to the set $\mathcal{O}$ defined in \eqref{def:setO}, i.e. away from the grazing set, where the number of reflections that any trajectory undergoes within a finite interval of time is uniformly bounded as a consequence of Lemma \ref{lem:regularitycharacflow}. By convergence of $E^n \to E,$ the characteristic flow $(X^n(s;t,x,v),V^n(s;t,x,v))$ converges uniformly to $(X(s;t,x,v),V(s;t,x,v)),$ and therefore $Q^n(t) \to Q(t)$ uniformly on $[0,T]$, see e.g. \cite[Proposition 5]{HwangVelazquez09} for details. Therefore, there exists $n_0>0$ such that for all $n\geq n_0$ the assumption $Q^n(t)<K$ in Proposition \ref{prop:uniformcontrols} can be removed as it follows from Proposition \ref{prop:Pfaffelmoser} and our choice of initial data. Hence, in order to prove that our construction holds as $T\to \infty$ it is enough to show that $K(T)$ given by Proposition \ref{prop:Pfaffelmoser} is finite for any $T>0$, which follows classically from our bound \eqref{eq:bounddotQ} on the discrete derivative of $Q$. 
\paragraph{Uniqueness: } 
Let us consider $(f^1,E^1)$ and $(f^2,E^2)$ two solutions of the Vlasov-Poisson system in the sense of Theorem \ref{thm:mainVP} with the same initial condition $f_0$. We easily see that the difference $f^1-f^2$ satisfies
\begin{align*}
\pa_t (f^1-f^2) + v\cdot \na_x (f^1-f^2) + E^1 \cdot \na_v (f^1-f^2) = (E^1-E^2)\cdot \na_v f^2. 
\end{align*}
We then consider the characteristic flow $(X^1(s;t,x,v),V^1(s;t,x,v))$ associated with the force field $E^1$, which is indeed characteristic for the transport operator on the left-hand-side of the equation above, and integrate along this flow to write
\begin{align*}
(f^1-f^2) (t,x,v) = \int_0^t (E^1-E^2)(s,X^1(s;t,x,v)) \cdot \na_v f^2 (s,X^1(s;t,x,v),V^1(s;t,x,v)) \d s 
\end{align*}
since $(f^1-f^2) (0,x,v)=0 $ by assumption. Using classical estimates on the Poisson kernel in a $C^{2,1}$ domain $\Omega$ we have for all $(s,x)$:
\begin{align} \label{eq:uniquessE1E2}
|(E^1-E^2)(s,x) |   \lesssim \int_\Omega \frac{ |\rho^1(s,y) -\rho^2(s,y) |}{|x-y|^2} \d y 
\end{align}
hence, with Fubini and the fact that the characteristic flow is Hamiltonian 
\begin{align*}
\| (f^1-f^2)(t) \|_{L^1(\Omega\times\RR^3)} \leq \int_0^t \int_\Omega | \rho^1(s,y)-\rho^2(s,y)| \iint_{\Omega\times\RR^3} \frac{ |\na_v f^2(s,x,v)| }{|x-y|^2} \d x \d v \d y \d s .
\end{align*}
On the one hand we have
\begin{align*}
\iint_{\Omega\times\RR^3} \frac{ |\na_v f^2(s,x,v)| }{|x-y|^2} \d x \d v &\leq \int_{|y-x|<1} \frac{ | \na_vf^2(t,y) |}{|x-y|^2} \d y  + \int_{|y-x|\geq 1} \frac{ | \na_vf^2 (t,y)| }{|x-y|^2} \d y \\
&\lesssim  \| \na_vf^2(t) \|_{L^\infty (\Omega;L^1(\RR^3))} +  \| \na_vf^2(t) \|_{L^1(\Omega\times\RR^3)}
\end{align*}
where both norms of $\na_vf^2$ are uniformly bounded on $[0,T]$ by assumption. On the other hand, we have 
\begin{align*}
\int_\Omega |\rho^1(s,y) -\rho^2(s,y) | \d y \leq \| (f^1-f^2)(s) \|_{L^1(\Omega\times\RR^3)} 
\end{align*}
hence
\begin{align*}
\| (f^1-f^2)(t) \|_{L^1(\Omega\times\RR^3)} \leq C(T) \int_0^t \| (f^1-f^2)(s) \|_{L^1(\Omega\times\RR^3)}  \d s 
\end{align*}
and uniqueness follows from Gr\"{o}nwall's Lemma, which concludes the proof of Theorem \ref{thm:mainVP}. 
\end{proof}

\section{The ionic Vlasov-Poisson system}  \label{sec:VPME} 

We now turn to the VPME system \eqref{eq:VPME}. The general strategy of proof for Theorem \ref{thm:mainVPME} is the same as for Theorem \ref{thm:mainVP} in the Vlasov-Poisson system case, with the exception of the elliptic regularity estimate of Corollary \ref{cor:iterative} which do not apply to the non-linear Poisson equation of \eqref{eq:VPME}. Therefore, we will prove analogous elliptic estimates in the next two sections as stated in Proposition \ref{prop:EllipticRegDir} for the Dirichlet case, and Proposition \ref{prop:EllipticRegNeu} for the Neumann case and in Section \ref{sec:proofthmVPME} we will rather briefly present the proof of Theorem \ref{thm:mainVPME}, focusing on the differences between the VP and the VPME case in order to avoid unnecessary repetitions.

\subsection{The Dirichlet case} 

\begin{prop}\label{prop:EllipticRegDir}
For any $\rho \geq 0$ in $C^{0,\alpha}(\Omega)$, the non-linear Poisson equation with Dirichlet boundary condition 
\begin{equation}\label{eq:PMEDir}
\left\{ \begin{aligned} 
& \Delta U = e^U - \rho - 1 \qquad & x\in\Omega \\
& U(x) = 0 & x\in\dO .
\end{aligned} \right. 
\end{equation}
has a unique solution $U\in H^1_0(\Omega)$. Furthermore, this solution is in $C^{2,\alpha}(\Omega)$ and satisfies $\pa_n U(x) <0$ for all $x\in\dO$.
\end{prop}
\begin{proof}
In the spirit of \cite[Proposition 3.5]{GriffinIacobelli21torus}, we prove existence of a solution in $H^1_0(\Omega)$ via a Calculus of Variation approach, by finding a minimiser for the energy functional 
\begin{align*}
\phi \to \mathcal{E}_D[\phi] := \int_{\Omega} \left( \frac12 |\na \phi|^2 + e^\phi - \phi - \rho \phi \right) \d x 
\end{align*}
in $H^1_0(\Omega),$ and proving that this minimiser solves the Euler-Lagrange equation associated with $\mathcal{E}_D$, which is \eqref{eq:PMEDir}. Uniqueness then follows from the strict convexity of $\mathcal{E}_D$.\\
First of all, for any $\phi\in H^1_0(\Omega)$ let us write $\phi_+(x) = \max(0,\phi(x))$. Since $\rho\geq 0$, we have $-\rho \phi \geq -\rho \phi_+$ and $e^\phi-\phi \geq e^{\phi_+} - \phi_+$ because for all $x\leq 0$, $e^x-x \geq e^0-0 = 1$. As a consequence 
\begin{align*}
\mathcal{E}_D[\phi] &= \int_\Omega \left( \frac{|\na \phi_+|^2}{2} + \frac{|\na(\phi-\phi_+)|^2}{2} + e^\phi -\phi -\rho \phi \right)\d x \\
&\geq \int_\Omega \left( \frac{|\na \phi_+|^2}{2} + e^{\phi_+}-\phi_+ - \rho \phi_+ \right) \d x = \mathcal{E}_D(\phi_+) .
\end{align*}
Hence, if $\phi$ minimises $\mathcal{E}_D$ on $H^1_0(\Omega)$ then $\phi=\phi_+$. Moreover, by the strong maximum principle, if $\rho \not\equiv 0$ then $\phi>0$ in $\Omega$ and cancels at the boundary, which implies $\pa_n \phi <0$ on $\dO$. \\
Secondly, we show existence of a minimiser. Let us consider a minimising sequence $(\phi^k)$, i.e. a sequence in $H^1_0(\Omega)$ such that
\begin{align*}
\mathcal{E}_D[\phi^k] \to \inf_{\phi\in H^1_0(\Omega)} \mathcal{E}_D[\phi] =: m .
\end{align*}
Note that since $\mathcal{E}_D[\phi_+]\leq \mathcal{E}_D[\phi]$ we can assume w.l.o.g. that $\phi^k\geq 0$. We want to show that the sequence $\phi^k$ is uniformly bounded in $H^1_0(\Omega)$. To that end, we first notice that $\mathcal{E}_D(0)= |\Omega|$ hence, for $k$ large enough,
\begin{align*}
\mathcal{E}_D[\phi^k] \leq |\Omega|. 
\end{align*}
Furthermore, since $\phi^k\geq0$, $e^{\phi^k}-\phi^k \geq (\phi^k)^2$ and since $\rho\in L^\infty(\Omega)$ we can fix $C>0$ such that $\frac{(\phi^k)^2}{2} \geq \rho \phi^k - C$. As a result, we have
\begin{align*}
|\Omega| \geq \mathcal{E}_D[\phi^k] \geq \int_\Omega \left( \frac{|\na \phi^k|^2}{2} + \frac{(\phi^k)^2}{2} -C \right) \d x = \frac12 \| \phi^k \|_{H^1_0(\Omega)} - C|\Omega|, 
\end{align*}
hence the sequence $(\phi^k)$ is equibounded in $H^1_0(\Omega)$. The rest of the proof of existence of a solution to \eqref{eq:PMEDir} follows exactly the proof of \cite[Proposition 3.5]{GriffinIacobelli21torus}, which is a similar result in the torus. For the sake of completeness we outline the main arguments here. The boundedness of $(\phi^k)$ in $H^1_0(\Omega)$ implies that, up to a subsequence, $\phi^k \to U$ a.e., with $U$ a minimiser of $\mathcal{E}_D$. Finally, one can show that $U$ solves the Euler-Lagrange equation associated with $\mathcal{E}_D$, namely \eqref{eq:PMEDir}, by investigating the limit as $\eta \to 0$ of $(\mathcal{E}_D[U+\eta \varphi] - \mathcal{E}_D[U])/\eta$ for any $\varphi \in C^\infty_c(\bO)$. \\
We now turn to the regularity of $U$. We split $U$ into a regular part $\hu$ and a singular part $\bu$ solutions to 
\begin{equation} \label{def:hubuDir} 
\left\{ \begin{aligned} 
& \Delta \hu = e^{\bu +\hu} - 1 \\
& \hu|_{\dO} = 0 
\end{aligned} \right. 
,\qquad 
\left\{ \begin{aligned} 
& \Delta \bu = -\rho\\
& \bu|_{\dO} = 0 .
\end{aligned} \right. 
\end{equation} 
By classical elliptic PDE theory, see e.g. \cite[Chapter 6]{Evans10} on a $C^{2,1}$ domain $\Omega$, we know that for $\rho\in C^{0,\alpha}(\Omega)$ there is a unique solution $\bu \in C^{2,\alpha}_c(\bO) \subset H^1_0(\Omega)$, and consequently we also have a unique $\hu = U - \bu \in H^1_0(\Omega)$. We are now interested in the regularity of $\hu$. \\
First of all, we know that $\hu$ is the unique minimiser in $H^1_0(\Omega)$ of 
\begin{align*}
\phi \to \hat{\mathcal{E}}_D [\phi] := \int_\Omega \left( \frac12 |\na\phi|^2  + e^{\bu+\phi} - \phi\right) \d x 
\end{align*}
where the Dirichlet Poisson potential $\bu$ is uniformly bounded, i.e. there exists $M_1>0$ such that $-M_1\leq \bu(x) \leq M_1$ for all $x\in\bO$. By minimality, this means on the one hand
\begin{align*}
\hat{\mathcal{E}}_D[\hu] \leq \hat{\mathcal{E}}_D[0] = \int e^{\bu} \d x = e^{M_1} |\Omega| <\infty.
\end{align*}
On the other hand, using the Poincaré inequality 
\begin{align*}
\hat{\mathcal{E}}_D[\hu] \geq \frac12 \int_\Omega ( |\na\hu|^2 + e^{\bu+\hu})\d x - C\| \na\hu\|_{L^2(\Omega)} \geq \int_\Omega e^{\hu+\bu} \d x - C^* 
\end{align*}
for some $C^*>0$. Combining the two estimates we get a bound of $e^\hu$ in $L^1(\Omega)$:
\begin{align*}
\int_\Omega e^\hu \d x \leq e^{M_1} (e^{M_1} |\Omega| + C^*).
\end{align*}
Furthermore, by construction we know that for any test function $\phi\in H^1_0(\Omega)$, $\hu$ satisfies
\begin{align*}
\int_\Omega \left( \na \hu \cdot \na \phi + (e^{\bu+\hu}-1)\phi \right)\d x  = 0 .
\end{align*}
In particular for any $n\in \mathbb{N}$, writing $\hu_n:=\hu\wedge  n$, we can take $\phi_n = e^{\hu_n} -1 $ which is indeed in $H^1_0(\Omega)\cap L^\infty(\Omega)$, hence 
\begin{align*}
\int_\Omega \left( |\na\hu|^2 e^\hu \mathds{1}_{\hu\leq n} + e^\bu e^{\hu+\hu_n} - e^{\hu+\bu} - e^{\hu_n}+1 \right) \d x = 0 .
\end{align*}
Since $|\na\hu|^2e^\hu +1 \geq 0$, this yields
\begin{align} \label{eq:estimateUnDir}
\int_\Omega e^\bu e^{\hu+\hu_n} \d x \leq \int_\Omega( e^{\hu+\bu} + e^{\hu_n} )\d x .
\end{align}
Moreover, by construction $e^{\hu_n}$ is increasing and converges to $e^{\hu}$ hence the monotone convergence theorem yields 
\begin{align*}
\| e^\hu \|_{L^2(\Omega)}^2 = \int_\Omega e^{2\hu} \d x \leq \frac{1+e^{M_1}}{e^{-M_1}} \int_\Omega e^\hu \d x  := C_0 \| e^\hu\|_{L^1(\Omega)} .
\end{align*}
We may iterate this argument: for any $k\in\mathbb{N}$ we  take $\phi = e^{k\hu_n} -1 $ and write $C_0=(1+e^{M_1})e^{M_1}$, we get
\begin{align*}
\|e^\hu\|_{L^k(\Omega)}^k = \int_\Omega e^{k\hu} \d x \leq C_0 \int_\Omega e^{(k-1)\hu} \d x \leq \dots \leq  C_0^{k-1} \| e^\hu\|_{L^1(\Omega)} .
\end{align*}
Thus, for any $k\in\mathbb{N}$ we have $\Delta \hu = e^{\hu+\bu}-1 \in L^k(\Omega)$ with $\| e^{\hu+\bu}-1 \|_{L^k(\Omega)} \leq Ce^{5M_1}$, so by standard elliptic regularity $\hu \in W^{2,k}(\Omega)$, see e.g. \cite[Section 6.3]{GilbargTrudinger}. We can take $k$ large enough for Sobolev embedding to yield $\hu\in C^{1,\alpha}(\Omega)$ with
\begin{align} \label{eq:controlnahuDir}
\| \na \hu \|_{C^{0,\alpha}(\Omega)} \leq C e^{5M_1}
\end{align}
which in turns implies $e^U-\rho-1 \in C^{0,\alpha}(\Omega)$ since $\bu\in C^{2,\alpha}_c(\bO)$. Standard elliptic regularity for \eqref{eq:PMEDir} then yields $U \in C^{2,\alpha}_c(\bO)$ and concludes the proof since we already have $\pa_n U<0$ on $\dO$ from the fact that $U$ minimises $\mathcal{E}_D$. 

\end{proof}

\subsection{The Neumann case}

\begin{prop} \label{prop:EllipticRegNeu}
	For any $\rho \geq 0$ in $C^{0,\alpha}(\Omega)$ with $\int_\Omega \rho \d x =1$, consider the non-linear Poisson equation with Neumann boundary condition 
	\begin{equation}\label{eq:PMENeu}
	\left\{ \begin{aligned} 
	& \Delta U = e^U - \rho - 1 \qquad & x\in\Omega \\
	& \pa_nU(x) = h & x\in\dO .
	\end{aligned} \right. 
	\end{equation} 
	with $h<0$ in $C^{1,1} (\dO)$ satisfying \eqref{hyp:hPMENeu}. There is a unique solution $U\in H^1(\Omega)$ to this problem. Furthermore, this solution is in $C^{2,\alpha}(\Omega)$. 
\end{prop}

Note that one could remove the assumption $\int_\Omega \rho \d x =1$ as long as $\rho$ remains integrable on $\Omega$. The condition on $h$ \eqref{hyp:hPMENeu} would then be 
$$ h<0, \qquad \int_{\dO} |h| \d \sigma(x) < \int_\Omega \rho \d x + |\Omega|. $$
\begin{proof}
Analogously to the Dirichlet case, we will prove existence of solution to \eqref{eq:PMENeu} in $H^1(\Omega)$ by finding the minimiser in $H^1(\Omega)$ of
\begin{align*}
\mathcal{E}_N[U] := \int_\Omega \left( \frac12 |\na U|^2 + e^U - U - \rho U  \right) \d x - \int_{\dO} U h \d \sigma(x)
\end{align*}
and uniqueness follows from the strict convexity of $\mathcal{E}_N$. \\
Let us consider a minimising sequence $(\phi^k)$ in $H^1(\Omega)$ and prove that it is equibounded. An upper bound is immediate since $\mathcal{E}_N[0]=|\Omega|$, hence for $k$ large enough $\mathcal{E}_N[\phi^k]\leq |\Omega|$. Note however that, unlike the Dirichlet case, we cannot easily compare $\mathcal{E}_N[\max(\phi,0)]$ with $\mathcal{E}_N[\phi]$. Instead, we will show equiboundedness of  the positive and negative parts of $\phi^k$ which we denote $\phi^k_\pm = \pm\max(0,\pm \phi^k)$ and that will yield the equiboundedness of $\phi^k$ since $\mathcal{E}_N [\phi^k] = \mathcal{E}_N[\phi^k_+] + \mathcal{E}_N[\phi^k_-]$ and $\| \phi^k\|_{H^1(\Omega)} = \| \phi^k_+\|_{H^1(\Omega)} + \| \phi^k_-\|_{H^1(\Omega)} $.\\
On the one hand, we have 
\begin{align*}
- \int_{\dO}  h\phi^k_+ \d \sigma(x) \geq 0
\end{align*}
and, with the same arguments as in the Dirichlet case, we have $e^{\phi^k_+} - \phi^k_+ - \rho \phi^k_+ \geq \frac12 (\phi^k_+)^2 - C$ for some $C>0$, hence 
\begin{align*}
|\Omega| \geq \mathcal{E}_N[\phi^k_+] \geq \int_\Omega \left( \frac12 |\na\phi^k_+|^2 + \frac12 (\phi^k_+)^2 - C\right) \d x  \geq \frac12 \|\phi^k_+\|_{H^1(\Omega)} - C |\Omega|
\end{align*}
which is the equiboundedness  of $\phi^k_+$.  \\
On the other hand, we write $\phi^k_- = C_k + \psi_k$ with $C_k = \int_\Omega \phi^k_- \d x \leq 0$ and $\psi_k = \phi^k_- - C_k \in H^1(\Omega)$ with $\int_\Omega \psi_k \d x = 0$, which yields
\begin{align*}
\mathcal{E}_N[\phi^k_-] &= \int_\Omega \left( \frac12 |\na\psi_k|^2 +e^{C_k+\psi_k} - (\rho+1) \psi_k  \right) \d x - \int_{\dO} h \psi_k \d \sigma(x)  - C_k \left( 1 + |\Omega|  + \int_{\dO} h \d \sigma(x) \right) .
\end{align*}
Assumption \eqref{hyp:hPMENeu} immediately yields for some $c_0>0$
\begin{align} \label{eq:Ckhyph}
- C_k \left( 1 + |\Omega|  + \int_{\dO} h \d \sigma(x) \right)  \geq c_0 |C_k|. 
\end{align}
On the other hand, using the Poincaré inequality and the Sobolev trace theorem we have
\begin{align*}
\int_\Omega \left( \frac12 |\na\psi_k|^2 +e^{C_k+\psi_k} - (\rho+1) \psi_k  \right) \d x - \int_{\dO} h \psi_k \d \sigma(x) &\geq \frac12 \int_\Omega |\na\psi_k|^2 \d x - C\int_\Omega |\psi_k| \d x - C \int_{\dO} |\psi_k| \d \sigma(x) \\
 &\geq \frac12 \int_\Omega |\na\psi_k|^2\d x - C\|\na  \psi_k\|_{L^2(\Omega)}\\ &\geq \frac12 \int_\Omega | \na\psi_k|^2 \d x - C^*
\end{align*}
for some $C^*>0$. Together with \eqref{eq:Ckhyph} and the upper bound $\mathcal{E}_N(\phi_k)\leq |\Omega|$ we have 
\begin{align*}
\|\na \phi^k_- \|^2_{L^2(\Omega)} + c_0 |C_k| \leq C^*+|\Omega| .
\end{align*}
Finally, using the Poincaré inequality again we get boundedness of $\phi^k_-$ in $H^1(\Omega)$ since
\begin{align*}
\| \phi^k_-\|_{L^2(\Omega)} \leq \| \phi^k_- - C_k \|_{L^2(\Omega)} + |C_k| |\Omega| \leq \| \na \phi^k_-\|_{L^2(\Omega)} + C 
\end{align*}
and the equiboundedness of $\phi_k$ in $H^1(\Omega)$ follows. This implies, up to a subsequence, convergence a.e. of $\phi^k$ towards $U$, the unique minimiser of $\mathcal{E}_N$ in $H^1(\Omega)$. Let us now check that the Euler-Lagrange equation associated with $\mathcal{E}_N$ is indeed \eqref{eq:PMENeu}. For any $\eta>0$ and $\phi\in C^\infty(\bO)$, since $U$ is the minimiser of $\mathcal{E}_N$ we have $\mathcal{E}_N[U+\eta\phi] \geq \mathcal{E}_N[U]$ and 
\begin{align*}
0&\leq \lim_{\eta\to 0} \frac{1}{\eta} \big( \mathcal{E}_N[U+\eta\phi] -\mathcal{E}_N[U]\big) \\
&\leq \lim_{\eta \to 0} \left[ \int_\Omega \left( \na U\cdot \na \phi + \frac12 \eta |\na \phi|^2 + e^U \left( \frac{e^{\eta \phi}-1}{\eta} \right)- \phi - \rho\phi \right) \d x - \int_{\dO} \phi h \d \sigma(x) \right]  \\
&\leq \int_\Omega \left[ \na U \cdot \na \phi + (e^u-1-\rho)\phi \right] \d x - \int_{\dO} \phi h \d \sigma(x) .
\end{align*}
This holds for any $\phi \in C^\infty(\bO)$, so it is true in particular for $\tilde\phi = -\phi$ and, hence for all $\phi\in C^\infty(\bO)$, $U$ satisfies 
\begin{align*}
\int_{\Omega} \left[ \na U \cdot \na \phi + (e^{U}-1-\rho)\phi \right] \d x - \int_{\dO} \phi h \d \sigma(x) =0
\end{align*}
which means that $U$ is indeed the unique weak solution in $H^1(\Omega)$ of \eqref{eq:PMENeu}. \\
For the regularity of $U$, we follow the strategy of the Dirichlet case. We split $U$ into a regular part $\hu$ and a singular part $\bu$ solutions to 
\begin{equation} \label{def:hubuNeu} 
\left\{ \begin{aligned} 
& \Delta \hu = e^{\bu +\hu} - 1 & \mbox{ in } \Omega \\
& \pa_n \hu= h_1 & \mbox{ on } \dO
\end{aligned} \right. 
,\qquad 
\left\{ \begin{aligned} 
& \Delta \bu = -\rho & \mbox{ in } \Omega\\
& \pa_n \bu= h_2 & \mbox{ on } \dO.
\end{aligned} \right. 
\end{equation} 
with 
\begin{equation*}
\left\{ \begin{aligned} 
& h_2< 0 , \quad \int_{\dO} h_2 \d \sigma(x) = -1, \\
& h_1\leq 0 , \quad \int_{\dO} h_1 \d \sigma(x) = \int_{\dO} h \d \sigma(x) +1 
\end{aligned} \right. 
\end{equation*} 
and naturally $h_1 + h_2 = h$ on $\dO$. By standard elliptic regularity theory, see e.g. \cite{Nardi13}, there exists a unique (up to an additive constant) solution $\bu \in C^{2,\alpha}(\Omega)$ for all $\alpha\in(0,1)$, which in turns yields existence and uniqueness (for a fixed $\bu$) of $\hu \in H^1(\Omega)$ which satisfies for all $\phi \in H^1(\Omega)$ 
\begin{align*}
\int_\Omega \left( \na\hu \cdot \na \phi + (e^{\hu+\bu}-1)\phi \right) \d x - \int_{\dO} h_1 \phi \d \sigma(x) = 0.
\end{align*}
Next, we write $\hu_n = \hu \wedge n$ for any $n\in \mathbb{N}$ and take $\phi = e^{\hu_n} \in H^1(\Omega)\cap L^\infty(\Omega)$ as a test function. Since $h_1\leq 0$ we get a similar estimate as \eqref{eq:estimateUnDir} in the Dirichlet case: 
\begin{align*}
\int_\Omega ( e^{\bu} e^{\hu+\hu_n} - e^{\hu_n}) \d x = - \int_\Omega |\na\hu |^2 e^{\hu} \mathds{1}_{\hu\leq n} \d x - \int_{\dO} |h_1|e^{\hu_n} \d \sigma(x) \leq 0 
\end{align*}
hence, for $M_1>0$ such that $-M_1\leq \bu(x) \leq M_1$ for all $x\in \bO$ we have by monotone convergence 
\begin{align*}
\| e^{\hu}\|_{L^2(\Omega)}^2 \leq e^{M_1} \| e^{\hu}\|_{L^1(\Omega)} .
\end{align*}
We can iterate this estimate, choosing $\phi = e^{k\hu_n}$ as a test function for any $k\in\mathbb{N}$ and derive 
\begin{align*}
\| e^{\hu} \|_{L^k(\Omega)}^k \leq e^{M_1} \| e^{\hu} \|_{L^{k-1}(\Omega)}^{k-1} \leq \dots \leq e^{(k-1)M_1}  \| e^{\hu} \|_{L^1(\Omega)}.
\end{align*}
Furthermore, we can take $\phi = 1 \in H^1(\Omega)$ as a test function as well which yields 
\begin{align*}
\int_\Omega e^{\hu} \d x \leq e^{M_1} \int_\Omega (e^{\hu+\bu} -1)\d x \leq e^{M_1} \| h_1 \|_{L^1(\dO)}.
\end{align*}
As a consequence, $e^\hu \in L^k(\Omega)$ for all $k\in \mathbb{N}$ which yields $e^{\bu+\hu} -1 \in L^k(\Omega)$ with $\| e^{\bu+\hu}-1\|_{L^k(\Omega)}\leq C e^{3M_1}$. Similarly to the Dirichlet case, standard elliptic regularity yields $\hu \in W^{2,k}(\Omega)$ and we can take $k$ large enough for Sobolev embedding to yields $\hu\in C^{1,\alpha}(\Omega)$ with 
\begin{align} \label{eq:controlnahuNeu}
\| \na \hu \|_{C^{0,\alpha}(\Omega)} \leq C e^{3M_1} 
\end{align}
and since $\bu\in C^{2,\alpha}(\bO)$ this means $e^U -1-\rho \in C^{0,\alpha}(\Omega)$ and Proposition \ref{prop:EllipticRegNeu} follows by standard elliptic regularity theory. 
\end{proof}

\subsection{Proof of Theorem \ref{thm:mainVPME}} \label{sec:proofthmVPME}

\begin{proof}[Proof of Theorem \ref{thm:mainVPME}] Using the elliptic regularity estimates established above for both the Dirichlet and the Neumann case, we can now prove well-posedness of the VPME system in the sense of Theorem \ref{thm:mainVPME} following the same arguments as in the Vlasov-Poisson case. \\
More precisely, we notice that the electric fields $E=-\na U$, with $U$ solution to \eqref{eq:PMEDir} or \eqref{eq:PMENeu}, is $C^0_tC^{1,\mu}([0,T]\times\bO)^d$ and satisfies $E\cdot n(x) = -\pa_n  U > 0 $ for all $x\in\dO$ so the results of Section \ref{sec:velocitylemma} apply, namely there exists a $\delta$-kinetic distance $\alpha$ and we have a Velocity Lemma. Moreover, if we define an iterative sequence \eqref{eq:iterative} with the non-linear Poisson equation \eqref{eq:PMEDir} or \eqref{eq:PMENeu}, then Theorem \ref{thm:LinearVlasov} holds and the elliptic regularity estimates proved above ensure that both Corollary \ref{cor:iterative} and Proposition \ref{prop:uniformcontrols} hold. Therefore, we have constructed a classical solution of the VPME system under the assumption of uniformly bounded velocities $Q(t)<K(T)$ for all $t\in[0,T]$, with $Q(t)$ given by \eqref{def:Qt}. In order to remove this assumption we decompose the electric field as $E=\hat{E} + \bar{E}$ with the regular part given by $\hat{E} = -\na \hu$ and the singular part given by $\bar{E} = -\na \bu$ where $\hu$ and $\bu$ are defined either by \eqref{def:hubuDir} or \eqref{def:hubuNeu}. \\
On the one hand, since $|\Omega|<\infty$ and by classical elliptic regularity, see e.g. \cite[Chapter 6]{GilbargTrudinger} or \cite{Nardi13},
\begin{align*}
\| \bu \|_{L^\infty(\Omega)} \leq C (1+ \|\rho\|_{C^{0,\alpha}(\Omega)} )
\end{align*}
where $C$ depends on $\|h\|_{C^{1,\alpha}(\dO)}$. Using \eqref{eq:controlnahuDir} or \eqref{eq:controlnahuNeu} this yields a uniform control of $\hat{E}$ in $L^{\infty}(\Omega)$ as 
\begin{align*}
\| \hat{E}(t,\cdot)\|_{L^\infty(\Omega)} \leq C\exp\big( (1+ \|\rho\|_{C^{0,\alpha}(\Omega)} )\big).
\end{align*} 
Therefore, in order to bound the maximum velocity $Q(t),$ one only needs to consider $\bar{E}$ for which the analysis developed in Section \ref{subsec:boundonQt}  applies, so we have an equivalent of Proposition \ref{prop:Pfaffelmoser} for VPME. Finally, we conclude the proof of Theorem \ref{thm:mainVPME} by the same argument as in Section \ref{sec:proofthmVP} in order to show  global existence and uniqueness, with the following estimate instead of \eqref{eq:uniquessE1E2}:
\begin{align*}
	|(E^1-E^2)(s,x) |   &\leq | (\bar{E}^1-\bar{E}^2)(s,x) | + | (\hat{E}^1-\hat{E}^2)(s,x) | \\
	&\lesssim \int_\Omega \frac{ |\rho^1(s,y) -\rho^2(s,y) |}{|x-y|^2} \d y +  \| \hat{E}^1(t,\cdot) - \hat{E}^2(t,\cdot)\|_{L^\infty(\Omega)}.
\end{align*}

\end{proof}

\appendix
\section{Numerical simulations} 

In this appendix, we present some numerical simulations of the trajectories of particles in the linear Vlasov equation \eqref{eq:linVlasov} in order to illustrate the results of Section \ref{sec:velocitylemma}. \\
Let us consider $\Omega$ to be the unit disk in dimension 2 and introduce an electric field $E$. In order to stay close to a Vlasov-Poisson framework, we will consider a field given by a Poisson equation with a nice density $\rho$. More precisely, we consider a Gaussian bell function centred at $x_0=(0.5,0)$
\begin{align*}
	\rho (x) = C\exp\left( -\frac{1}{1-4|x-x_0|^2} \right) \mathds{1}_{|x-x_0|<1/2} 
\end{align*}
and the stationary electric field $E$ given by $E=-\na U$, $\Delta U = -\rho$ with Dirichlet boundary conditions $U|_{\dO} = 0$. This field can be explicitly written as a convolution with the Green function of the ball, see e.g. \cite[Section 2.2.4]{Evans10}. We illustrate $\rho$ and $E$ in Figure \ref{fig:DensityandField} and note that $E$ is indeed outgoing on $\dO$ in the sense $E\cdot n(x) >0$ for all $x\in\dO$ as assumed in Section \ref{sec:velocitylemma}. 
\begin{figure}[!htb]
	\centering
	\caption{Choice of density $\rho$ and associated field $E$}
	\includegraphics[trim=100 100 0 200, clip, width=15cm, keepaspectratio]{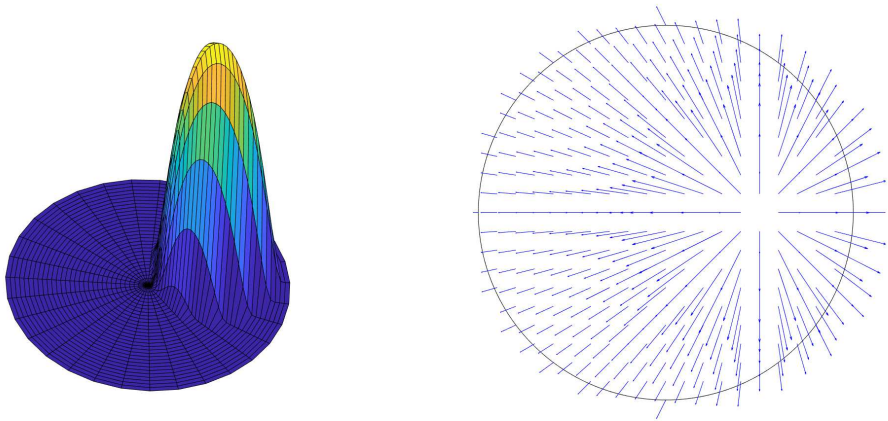}
	\label{fig:DensityandField}
\end{figure}

Given this fixed field $E$, we can compute the trajectory of a particle by solving the system of ODEs \eqref{eq:characX}-\eqref{eq:characV}-\eqref{eq:characSR}. For instance, if we consider a particle starting at $x=(-0.9,0)$ with velocity direction $(-0.2,1)$ then it's trajectory in the domain $\Omega$ will be given by Figure \ref{fig:Traj1}. 
\begin{figure}[!htb]
	\centering
	\caption{Trajectory in the disk from $x=(-0.9,0)$, $v\propto(-0.2,1)$ and evolution of speed}
	\includegraphics[trim=0 75 0 25, clip, width=15cm, keepaspectratio]{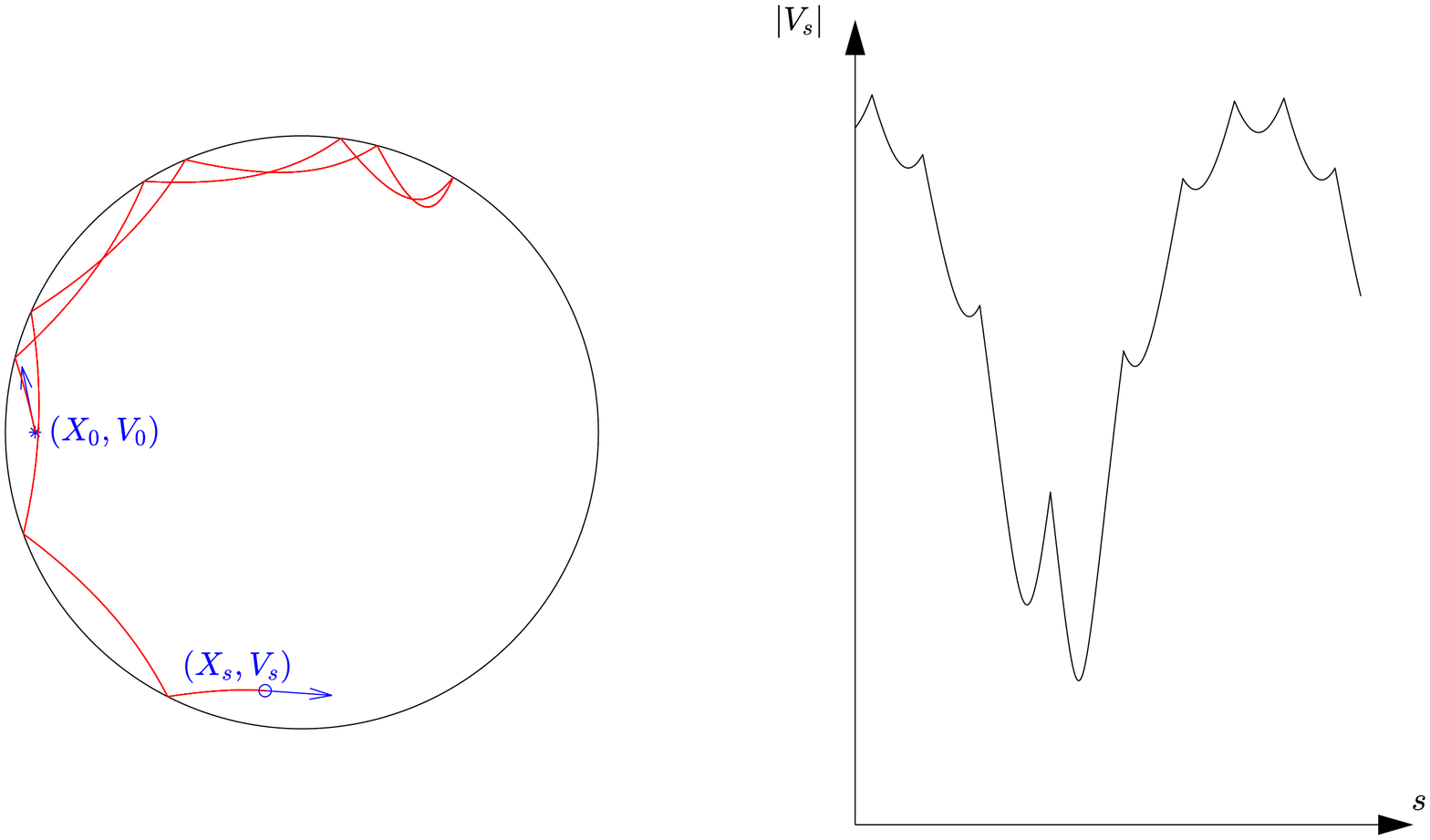}
	\label{fig:Traj1}
\end{figure}
On that figure, we also plot the norm of the velocity $s\to | V(s;0,x,v)|$ which decreases when the particle moves against the direction of the field, and increases when it follows the field. In particular, we see that when the trajectory moves towards the right of the disk, where the field is strongest, it may change direction if the norm of the velocity is too small, as is the case in Figure \ref{fig:Traj1}. This can be interpreted as the particle slowing down to the point where the electric field becomes stronger than the natural inertia of the particle, hence the change of direction. We also notice on this plot that the norm of the velocity is a continuous function of $s$, piece-wise smooth (for the regular field $E$ we consider in this example) with singularities at the points of reflection, as expected. \\
We would like to illustrate the Velocity Lemma \ref{lem:VelocityLemma}, which states that if a trajectory starts close to the grazing set $\gamma_0$ then it remains close to $\gamma_0$ through time, although the size of the neighbourhood increases exponentially fast c.f. \eqref{eq:isolationgamma0}. However, since any neighbourhood of $\gamma_0$ is a subset of the 4-dimensional phase-space, we cannot really illustrate this result on the plot of the trajectory (even though if a trajectory is close to $\gamma_0$ then necessarily $X_s$ is close to $\dO$ by construction but that is not a sufficient condition). Instead, let us look at the kinetic distance $\alpha$ given by \eqref{def:alphadeltakin}. In our example, we characterise the unit disk via the function $\xi(x) = \frac12 (|x|^2-1)$ for which the kinetic distance $\alpha$ can be written as
\begin{align*}
	\alpha(x,v) = \frac12 (v\cdot x)^2 + (1-|x|) \big(|v|^2+E(x)\cdot x\big). 
\end{align*}
Note that since the electric field is stationary, the kinetic distance does not depend directly on $t$. Introducing $\alpha(s) = \alpha(X(s;0,x,v),V(s;0,x,v))$ for a fixed $(x,v)\in\Omega\times\RR^2$ we plot in Figure \ref{fig:VL} the evolution of the kinetic distance along the trajectory illustrated in Figure \ref{fig:Traj1}. \\
The Velocity Lemma \ref{lem:VelocityLemma} gives a uniform exponential bound on $\alpha$: for all $(x,v)\in\bO\times\RR^2$ and $s\in(0,t)$: 
\begin{align*}
	\alpha(x,v) e^{-C_0[ (|v|+1)s + \|E\|_{L^\infty} s^2]} \leq \alpha(X(s;0,x,v),V(s;0,x,v))  \leq 	\alpha(x,v) e^{C_0[ (|v|+1)s + \|E\|_{L^\infty} s^2]}
\end{align*}
for some $C_0 = C_0( \xi,E) >0$. Since the constant $C_0$ is not given explicitly by our Velocity Lemma we will not illustrate this exponential bound. \\

\begin{figure}[!htb]
	\centering
	\caption{Kinetic distance along a trajectory}
	\includegraphics[trim=0 0 0 0, clip, width=7cm, keepaspectratio]{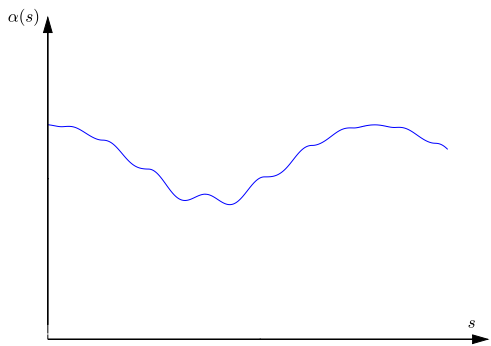}
	\label{fig:VL}
\end{figure}

To conclude this appendix we now consider other examples of trajectories and their associated kinetic distances in order to illustrate the variety of behaviour that these trajectories can exhibit and the associated variations of their kinetic distances. In Figure \ref{fig:MultipleTraj} we represent 6 trajectories, all starting with the same vertical direction of velocity $(0,1)$ (with a greater initial norm than the one of Figure \ref{fig:Traj1}, which is why there is not change of direction in plots 4 to 6 when the trajectory travels through the right side of the domain) and initial position on the $x-$axis with coordinate $-0.2, -0.4, -0.6,-0.8,-0.9,-0.95$ respectively. As in Figure \ref{fig:Traj1} we represented with a blue star the position at $s=0$ and with a blue circle the position at the end time, with colours matching that of Figure \ref{fig:MultipleKD}. These trajectories illustrate in particular the isolation of grazing. We see indeed that the trajectories 3 to 6, which start relatively close to the grazing set, remain in a neighbourhood of $\dO$, neighbourhood which grows smaller as $(x,v)$ grows closer to the grazing set $\gamma_0$.

\begin{figure}[!htb]
	\centering
	\caption{Examples of trajectories on the disk}
	\includegraphics[trim=50 150 0 125 , clip, width=15cm, keepaspectratio]{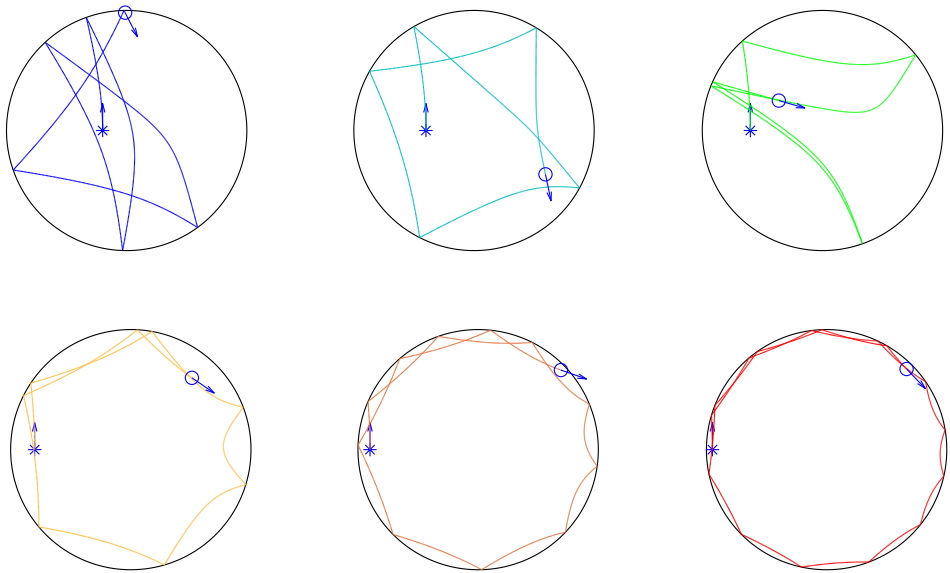}
	\label{fig:MultipleTraj}
\end{figure}
\begin{figure}[!htb]
	\centering
	\includegraphics[trim=50 150 0 140, clip, width=15cm, keepaspectratio]{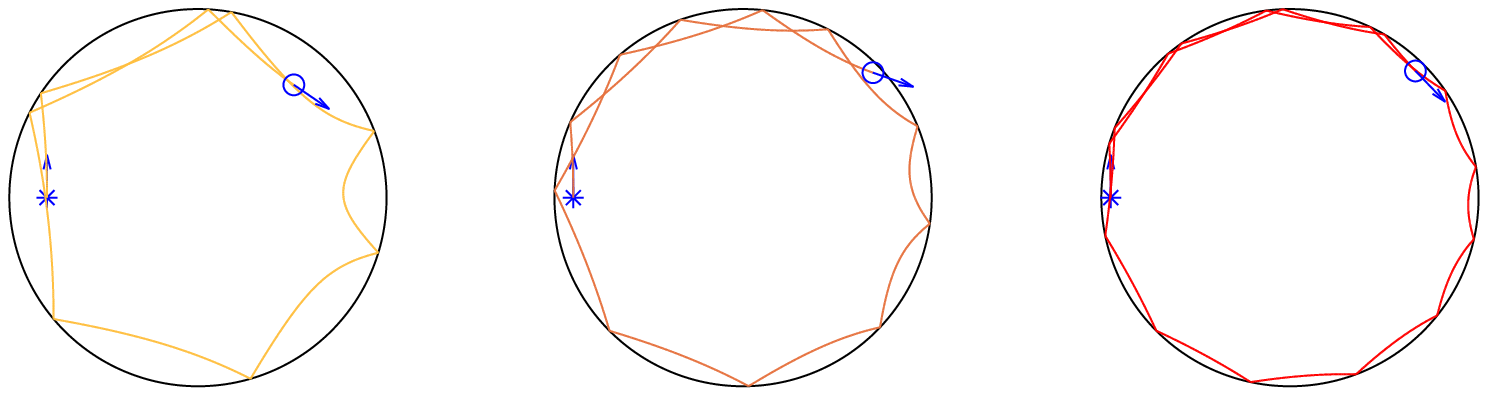}
\end{figure}

We also plot the evolution of the kinetic distance along these trajectories in Figure \ref{fig:MultipleKD}. Note that since they all start with the same velocity (and initial position on the same axis) it is easy to identify which curve corresponds to which trajectory by the initial value of the kinetic distance which decreases as the initial position grows near the boundary.  

\begin{figure}[!htb]
	\centering
	\caption{Kinetic distances as functions of $s$}
	\includegraphics[trim=0 0 0 0, clip, width=7cm, keepaspectratio]{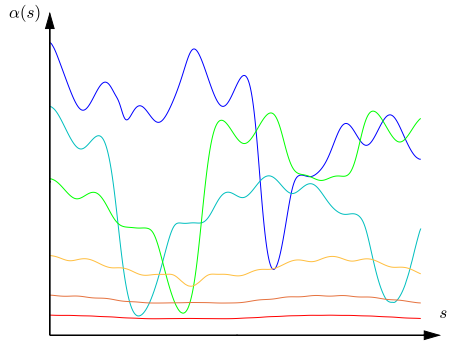}
	\label{fig:MultipleKD}
\end{figure}

One may observe many phenomena in this last illustration. For instance, we see that the kinetic distances of the last two trajectories, which start rather close to grazing, vary little through time. Note that this also applies to the trajectory of Figure \ref{fig:Traj1} which morally would fit between the 4th and 5th trajectories of Figure \ref{fig:MultipleTraj}. On the other hand, the kinetic distance of the trajectories that start with a position far from $\dO$ show significant variations, and we see in particular that, at their lowest, their value is close that of the last two trajectories. This illustrates the fact that if $x$ is close to $\dO$ and $|v| \ll 1$ then $\alpha(x,v)$ will be small even if $v$ is not tangential. The exponential bounds given by the Velocity Lemma \ref{lem:VelocityLemma} naturally allow for such behaviour since the exponential coefficients are uniform in $x\in\bO$ and only depend on the norm of $v$. \\

Finally, let us emphasis that for these illustrations we have chosen a very smooth electric field $E$ and a smooth domain $\Omega$ with constant curvature. Naturally, if the field $E$ is less regular, and if one consider a more general uniformly convex domain $\Omega$, then one may observe a much wider variety of behaviours for the trajectories of the linear Vlasov equation \eqref{eq:linVlasov}. On the importance of curvature, let us recall that we are not able yet to prove a strong enough isolation of the grazing set when the domain is not uniformly convex -- i.e. when the curvature may cancel pointwise or on portions of the boundary -- in order to conclude the Pfaffelmoser argument of Section \ref{subsec:boundonQt} since the exponential controls of our Velocity Lemma are significantly worse in the non-uniformly convex case, as explained in Remark \ref{rmk:non-unif}.

\bibliographystyle{siam}

\bibliography{biblio-1-1.bib}
\end{document}